\documentclass[11pt]{article}
\usepackage[numbers,sort&compress]{natbib}
\usepackage{enumerate}
\usepackage{amscd}
\usepackage{array}
\usepackage{amsmath}
\usepackage{latexsym}
\usepackage{amsfonts}
\usepackage{amssymb}
\usepackage{amsthm}
\usepackage{verbatim}
\usepackage{mathrsfs}
\usepackage{enumerate}
\usepackage{hyperref}

\theoremstyle{plain}
\theoremstyle{definition}\newtheorem{theorem}{Theorem}[section]
\theoremstyle{plain}\newtheorem{lemma}[theorem]{Lemma}
\theoremstyle{plain}\newtheorem{coro}[theorem]{Corollary}
\theoremstyle{plain}
\theoremstyle{remark}\newtheorem{remark}{Remark}[section]
\theoremstyle{definition}
\theoremstyle{plain}
\usepackage{xcolor}

\newcommand{\norm}[1]{\left\|#1\right\|}
\newcommand{\Div}{\mathrm{div}\,}
\newcommand{\B}{\Big}

\newcommand{\be}{\begin{equation}}
\newcommand{\ee}{\end{equation}}
 \newcommand{\ba}{\begin{aligned}}
 \newcommand{\ea}{\end{aligned}}

  \newcommand{\f}{\frac}
    
  \newcommand{\ben}{\begin{enumerate}}
   \newcommand{\een}{\end{enumerate}}

\newcommand{\ti}{\nabla}

\newcommand{\Rmnum}[1]{\expandafter\@slowromancap\romannumeral #1@}

\allowdisplaybreaks

\numberwithin{equation}{section}
\begin{document}%\begin{CJK*}{GBK}{fs}
%%%%%%%%%%%%%%%%%%%%%%%%%%%%%%%%%%%%%%%%%%%%%%%%%%%%%%%%%%%%%%%%%%%%%%%%%%%%%%%%%%%%%%%%%%%%%%%%%%%%
\title{ Gagliardo-Nirenberg inequality in anisotropic Lebesgue spaces and energy equality in the Navier-Stokes equations }
  \author{Yanqing Wang\footnote{College of Mathematics and Information Science, Zhengzhou University of Light Industry, Zhengzhou, Henan  450002,  P. R. China Email: wangyanqing20056@gmail.com}, ~~~~~Xue Mei\footnote{  College of Mathematics and Information Science, Zhengzhou University of Light Industry, Zhengzhou, Henan  450002,  P. R. China Email: meixue\_mx@outlook.com}\;~ ~and ~  Wei Wei\footnote{Corresponding author. School of Mathematics and Center for Nonlinear Studies, Northwest University, Xi'an, Shaanxi 710127,  P. R. China  Email: ww5998198@126.com }\;
   }

\date{}
\maketitle
 \begin{abstract}
 In this paper, it is shown that an analogue in mixed norm spaces of
  Gagliardo-Nirenberg inequality is valid. As an application, some new criteria
  for energy conservation of Leray-Hopf weak solutions to the 3D Navier-Stokes equations are established in these spaces, which extend the corresponding known results.
  \end{abstract}
%%%%%%%%%%
\section{Introduction}
\label{intro}
\setcounter{section}{1}\setcounter{equation}{0}
\subsection{Sobolev-Gagliardo-Nirenberg inequality}
Starting from the appearance of the Sobolev embedding theorem, it has become an essential
tool in modern theory of partial differential equations.
The classical Sobolev inequality reads
\be\label{Sobolevineq}
\|f\|_{L^{\f{rn}{n-r}}(\mathbb{R}^{n})}\leq C \|\nabla f\|_{L^{r}(\mathbb{R}^{n})},~~1\leq r<n.
\ee
%The has gained huge attraction (see \cite{[Chikami],[WWY],[HMOW],[Gagliardo],[Nirenberg]}   and references therein). In this direction,
A celebrated generalization of Sobolev inequality \eqref{Sobolevineq} is the following Gagliardo-Nirenberg inequality obtained independently by \cite{[Gagliardo]} and \cite{[Nirenberg]}
 \be\label{GNI}
\|D^{j}u\|_{L^{p}(\mathbb{R}^{n})}\leq C\|D^{m}u\|_{L^{r}(\mathbb{R}^{n})}^{\theta}\|u\|_{L^{q}(\mathbb{R}^{n})}^{1-\theta},
\ee
where $j,m$ are any integers satisfying  $0\leq j<m$, $~1\leq q,r\leq \infty,~$ and
\[
\frac{1}{p}-\frac{j}{n}= \theta(\frac{1}{r}-\frac{m}{n})+(1-\theta)\frac{1}{q}
\]
for all $\theta$ in the interval $\f{j}{m}\leq \theta\leq1,$ unless $1<r<\infty$ and $m-j-\f{n}{r}$ is a nonnegative integer.

Almost at the same time, Benedek-Panzone \cite{[BP]} introduced mixed  (anisotropic) Lebesgue spaces $L^{\overrightarrow{p}}(\mathbb{R}^{n})$ with $\overrightarrow{p}=(p_{1},p_{2},\cdots,p_{n})$
and extended the Sobolev inequality \eqref{Sobolevineq} to
\be\label{bpsobo}
\|f\|_{L^{\overrightarrow{p}}(\mathbb{R}^{n})}\leq C \|(-\Delta)^{\f{s}{2}} f\|_{L^{\overrightarrow{r}}(\mathbb{R}^{n})},~~ \text{with}~~ \sum_{i=1}^{n}\left(\frac{1}{r_{i}}-\frac{1}{p_{i}}\right)=s,~1<\overrightarrow{r}<\overrightarrow{p}<\infty.
\ee
It is clear that the anisotropic Lebesgue spaces $L^{\overrightarrow{p}}(\mathbb{R}^{n})$ are a generalization of usual Lebesgue spaces $L^{p}(\mathbb{R}^{n})$.
Later on, in the framework  of anisotropic  Lebesgue spaces and  anisotropic derivatives,  Besov-Il'in-Nikolski\v{i} \cite{[BIN]} showed an anisotropic  Gagliardo-Nirenberg inequality below
\be\ba\label{BIN}
&\|D^{\alpha}u\|_{L^{\overrightarrow{q}}(\mathbb{R}^{n})}\leq C\| u\|_{L^{\overrightarrow{p(0)}}(\mathbb{R}^{n})}^{\mu_{0}}\prod_{j=1 }^{n}
\|D^{l_{j}}_{j}u\|_{L^{\overrightarrow{p(j)}}(\mathbb{R}^{n})}^{\mu_{j}},~~\text{with}~~\sum_{j=0}^{n}\mu_{j}=1,~0\leq \mu_{j}\leq1,~  \\ &\alpha_{i}-\f1{q_{i}}=(l_{i}-\f{1}{p(i)_{i}})\mu_{i}-\sum_{0\leq j\leq n,\,j\neq i}\f{\mu_{j}}{p(j)_{i}}, ~1\leq i\leq n,~~~\text{and}~~~ \frac{1}{ q_{i} }\leq \sum^{n}_{j=0}\frac{\mu_{j}}{ {p(j)}_{i}}.
\ea\ee
 The Besov-Il'in-Nikolski\v{i} inequality \eqref{BIN}  is a very powerful tool and
 plays an important role in the study of partial differential equations (see \cite{[Zheng],[WW],[CT],[BYT],[BM],[GCS],[WWZ]} and references therein).
However, it seems that it is  difficult to derive from both of the inequality $\eqref{BIN}$ and its generalization in \cite{[Hytonen],[Esfahani]}  that
\be\label{spe1}
 \|u\|_{L_{x_{1}}^{4}L_{x_{2}}^{6}(\mathbb{R}^{2})}\leq C\|u\|_{L_{x_{1}}^{2}L_{x_{2}}^{4}(\mathbb{R}^{2})}^{\theta}\|\nabla u\|_{L_{x_{1}}^{2}L_{x_{2}}^{3}(\mathbb{R}^{2})}^{1-\theta},
\ee
owing to the condition $\eqref{BIN}_{2}$ involving each component of vector-valued indices.
In fact, the standard  Gagliardo-Nirenberg inequality
in mixed Lebesgue spaces without anisotropic derivatives such as $\eqref{spe1}$ is still unknown in the literature.
The main objective of this paper is to establish
the most general form of standard Gagliardo-Nirenberg inequality
in mixed Lebesgue spaces.  Our first result reads
\begin{theorem}\label{the1.1}
Let $0 \leq \sigma < s,$ $1 < \overrightarrow{p},\overrightarrow{q},\overrightarrow{r} < \infty,$ $0\leq\theta\leq 1,$ $\Lambda^{s}=(-\Delta)^{\f{s}{2}}$ be defined via
$\widehat{\Lambda^{s} f}(\xi)=|\xi|^{s}\hat{f}(\xi).$
Then there exists a positive constant $C$ independent of $u$ such that
 \be\label{glxied1}
\|\Lambda^{\sigma}u\|_{L^{\overrightarrow{p} } (\mathbb{R}^{n})} \leq C\|u\|_{L^{\overrightarrow{q}}(\mathbb{R}^{n}) }^{\theta}\|\Lambda^{s}u\|_{L^{\overrightarrow{r} } (\mathbb{R}^{n})}^{1-\theta},
\ee
if and only if
\be\label{condi}
\sum_{i=1}^{n}\frac{1}{p_{i}}-\sigma= \sum_{i=1}^{n}\frac{\theta}{q_{i}}+(1-\theta)\left(\sum_{i=1}^{n}\frac{1}{r_{i}}-s\right),~
\frac{1}{ p_{i} }\leq\frac{\theta}{ q_{i} }+\frac{1-\theta}{ r_{i} },
~0\leq\theta\leq 1-\frac{\sigma}{s}.
\ee
\end{theorem}
\begin{remark}
Gagliardo-Nirenberg inequality \eqref{glxied1} in mixed Lebesgue spaces
requires weaker restrictions than that of Besov-Il'in-Nikolski\v{i}'s anisotropic interpolation inequality \eqref{BIN}. In particular, the inequality
   \eqref{spe1} can be deduced from \eqref{glxied1} with $\theta=7/11$. Moreover, our result \eqref{glxied1} with $\theta=0$ improves the Sobolev inequality \eqref{bpsobo} in mixed Lebesgue spaces.
\end{remark}
As a byproduct, one can immediately obtain the following corollary, which was frequently  used in the regularity of the weak solutions of the Navier-Stokes equations (see \cite{[WWZ],[Zheng],[GCS]}).
\begin{coro}\label{coro1.3}
Let $q_{1},q_{2},q_{3}\in[2,\infty)$ and $\f12\leq \sum\limits_{i=1}^{3} \f{1}{q_i}\leq \f32$. Then there exists a positive constant $C=C(q_{1},q_{2},q_{3})$ such that
\begin{align}\label{zc}
\|u\|_{L^{\overrightarrow{q}}(\mathbb{R}^3)}
 \leq C\|\nabla u\|_{L^{2}(\mathbb{R}^{3})}^{\f32-{\sum_{i} \f{1}{q_i}}}\|u\|^{\sum_{i} \f{1}{q_i}-\f12}_{L^{2}(\mathbb{R}^{3})}.
\end{align}
\end{coro}
Since the first condition in \eqref{condi} for Gagliardo-Nirenberg inequality \eqref{glxied1} in anisotropic Lebesgue
 spaces  is posed indirectly on each component of indices, it seems that the
recent argument developed for  Fourier-Herz spaces in \cite{[Chikami]} and   Lorentz spaces in \cite{[WWY]}, which requires many applications of interpolation inequality \eqref{interi} involving each component, breaks down in the proof of Theorem \ref{the1.1}. To overcome this difficulty, we reformulate the desired  inequality \eqref{glxied1} as
$$
\| \Lambda^{\sigma}u \|_{L^{\overrightarrow{p} } (\mathbb{R}^{n})}\leq C\|\Lambda^{s(1-\theta)}u\|_{L^{\overrightarrow{\beta} } (\mathbb{R}^{n})}\leq C \|\mathcal{M} u \|^{\theta}_{L^{\overrightarrow{q} }(\mathbb{R}^{n})}\|\mathcal{M}(  \Lambda^{s } u)\|^{1-\theta}_{L^{\overrightarrow{r}}(\mathbb{R}^{n})},
~~ \frac{1}{\overrightarrow{\beta}}=\frac{\theta}{\overrightarrow{q}}+\frac{1-\theta}{\overrightarrow{r}}.
$$
Here the first inequality is partially  motivated by  \cite{[BIN],[Hytonen]}, and its proof relies on the celebrated
Marcinkiewicz-Lizorkin multiplier theorem together with a duality argument and Hardy-Littlewood-Sobolev inequality. The second inequality
rests on
  the following pointwise estimate proved by the decomposition of low and high frequencies
$$
 \left|\Lambda^{s(1-\theta)} u(x)\right| \leq C(\mathcal{M} u(x))^{\theta}\left(\mathcal{M} \Lambda^{s} u(x)\right) ^{1-\theta},
$$
whose proof can be found in \cite[p.84]{[BCD]}. This helps us to achieve the proof of Theorem \ref{the1.1}.

In the past decades,  important progress involving the Gagliardo-Nirenberg inequality \eqref{GNI} has been made in Besov spaces by Hajaiej-Molinet-Ozawa-Wang   \cite{[HMOW]}, in Besov-Fourier-Herz spaces by Chikami   \cite{[Chikami]}, and in Besov-Lorentz spaces by \cite{[WWY]} (see also Byeon-Kim-Oh \cite{[BKO]}). Inspired by this, a natural question arises whether Gagliardo-Nirenberg inequality \eqref{GNI} can be extended in the context of anisotropic Besov
  spaces.  Our corresponding result is formulated as follows.
 \begin{theorem}\label{the1.2}
   Assume that $0 \leq \sigma<s<\infty$ and $1 \leq \overrightarrow{q}, \overrightarrow{r}, \overrightarrow{p}\leq \infty.$ In addition, if
$\min\{\sum_{i=1}^{n}\frac{1}{q_{i}},\sum_{i=1}^{n}\frac{1}{r_{i}}\}<\sum_{i=1}^{n}\frac{1}{p_{i}}<\max\{\sum_{i=1}^{n}\frac{1}{q_{i}},\sum_{i=1}^{n}\frac{1}{r_{i}}\}$ or  $\sum_{i=1}^{n}\frac{1}{p_{i}}=\sum_{i=1}^{n}\frac{1}{r_{i}}<\sum_{i=1}^{n}\frac{1}{q_{i}}$, we require that
there exists $\alpha\in [0,1]$  such that
\be\label{addc}
\f{1}{\overrightarrow{p}}=\f{\alpha}{\overrightarrow{q}}+\f{1-\alpha}{\overrightarrow{r}}.
\ee
Then there holds for any $u \in \dot{B}_{\overrightarrow{r},  \infty}^{s}\left(\mathbb{R}^{n}\right) \cap \dot{B}_{\overrightarrow{q}, \infty}^{0}\left(\mathbb{R}^{n}\right)$ that
\be\label{glibxied}
\|u\|_{\dot{B}_{\overrightarrow{p},  1}^{\sigma}(\mathbb{R}^{n})} \leq C\|u\|_{\dot{B}_{\overrightarrow{q},  \infty}^{0}(\mathbb{R}^{n})}^{\theta}\|u\|_{\dot{B}_{\overrightarrow{r},  \infty}^{s}(\mathbb{R}^{n})}^{1-\theta},
\ee
with
$$
\sum_{i=1}^{n}\frac{1}{p_{i}}-\sigma= \sum_{i=1}^{n}\frac{\theta}{q_{i}}+(1-\theta)\left(\sum_{i=1}^{n}\frac{1}{r_{i}}-s\right),~s \neq \sum_{i=1}^{n}\B(\frac{1}{r_{i}}-\frac{1}{q_{i}}\B),~0<\theta<1-\frac{\sigma}{s}.
$$
 \end{theorem}

\begin{remark}
Under the condition \eqref{addc}, in view of the embedding relations
$\dot{B}_{\overrightarrow{p},  1}^{s}\hookrightarrow \dot{H}^{s}_{\overrightarrow{p} }\hookrightarrow\dot{B}_{\overrightarrow{p},  \infty}^{s}$ (see Lemma \ref{lem2.2}), it can be seen that Theorem \ref{the1.2} partially implies Theorem \ref{the1.1}. It is an interesting question how to remove the additional condition \eqref{addc}.
\end{remark}
The proof of Theorem \ref{the1.2}  is close to that in the recent work \cite{[Chikami],[WWY]}. The new ingredients are mainly the generalized Bernstein  type inequalities for anisotropic Lebesgue spaces  in  Lemma \ref{lem2.1}, which are of independent interest  and may be applied to other topics.

Next, we give an application of Gagliardo-Nirenberg inequality \eqref{glxied1} in Theorem  \ref{the1.1} in the study of energy conservation for weak solutions to the three-dimensional Navier-Stokes equations.
\subsection{An application of  Gagliardo-Nirenberg inequality in  anisotropic Lebesgue spaces to the Navier-Stokes system}
The classical  incompressible 3D  Navier-Stokes equations can be written as
\be\left\{\ba\label{NS}
&v_{t} -\Delta  v+ v\cdot\ti
v  +\nabla \Pi=0, \\
&\Div v=0,\\
&v|_{t=0}=v_0,
\ea\right.\ee
 where  the unknown vector $v=v(x,t)$ describes the flow  velocity field, and the scalar function $\Pi$ represents the   pressure.
 The  initial datum $v_{0}$ is given and satisfies the divergence-free condition.
 The global Leray-Hopf finite energy weak solution of the 3D  Navier-Stokes system has been  constructed   by Leray \cite{[Leray1]} for the Cauchy problem  and by Hopf \cite{[Hopf]} for Dirichlet  problem, and obeys the energy inequality
 $$
 \|v(T)\|_{L^{2}(\Omega)}^{2}+2 \int_{0}^{T}\|\nabla v\|_{L^{2}(\Omega)}^{2}ds\leq \|v_0\|_{L^{2}(\Omega)}^{2}.
 $$
   Unfortunately, both the conservation of energy
 and the regularity for this type of weak solutions are unknown. Here, we focus on the results involving energy  equality of the  Leray-Hopf weak solutions.
 The sufficient conditions for
  energy  equality  of weak solutions to the 3D Navier-Stokes equations \eqref{NS} have attracted a lot of attention (see \cite{[Lions],[Shinbrot],[BY],[CL],[Zhang],[WWY],[WY],[WMH]} and
  references therein). Concretely speaking,
  a weak solution to the Dirichlet problem on a bounded smooth domain $\Omega$ of  3D Navier-Stokes equations \eqref{NS} satisfies energy equality if one of the following four conditions holds
\begin{itemize}
 \item
 Lions \cite{[Lions]}:  $v\in L^{4}(0,T;L^{4}(\Omega));$
\item Shinbrot \cite{[Shinbrot]}:
$
v\in L^{p}(0,T;L^{q}(\Omega)),~\text{with}~\f{2}{p}+
 \f{2}{q}=1~\text{and}~q\geq 4;$
\item  Beirao da Veiga-Yang \cite{[BY]}:
$v\in L^{p}(0,T;L^{q}(\Omega)),$~with~ $\f{1}{p}+
 \f{3}{q}=1$~and~$3<q< 4;$
  \item Berselli-Chiodaroli  \cite{[BC]},  Beirao da Veiga-Yang  \cite{[BY2]},  Zhang \cite{[Zhang]}:
$
\nabla v \in L^{p} (0, T ; L^{q} (\Omega ) ),~\text{with}~
\frac{1}{p}+\frac{3 }{q}=2~\text{and}~\frac{3 }{2}<q<\frac{9}{5},~\text{or}~
\frac{1}{p}+\frac{6}{5 q}=1~\text{and}~q \geq \frac{9}{5}. $
   \end{itemize}
   In \cite{[WY]}, the authors observed that  all the  aforementioned  results on the  periodic domain $\mathbb{T}^{3} $ are a consequence of  the following
   energy conservation criterion via a combination of the velocity and the gradient of velocity
\be\label{wy}
  v\in L^{\f{2k}{k-1}} (0,T;L^{\f{2\ell}{\ell-1}}(\mathbb{T}^{3}) ) ~~  \text{and}~~ \nabla v \in L^{k} (0,T; L^{\ell}(\mathbb{T}^{3}) ).
 \ee
A parallel result of energy conservation criterion \eqref{wy}
 for both the Cauchy problem and Dirichlet problem is given in \cite{[WMH]}.  We also refer the readers to \cite{[CL],[CCFS],[WWY],[WMH]} for the    sufficient conditions keeping the energy of Navier-Stokes equations on the whole space $\mathbb{R}^{3}$.

Finally, we present some new energy conservation criteria in anisotropic Lebesgue spaces   via the Gagliardo-Nirenberg inequality \eqref{glxied1}.
 \begin{theorem}\label{the1.3} The Leray-Hopf weak solutions to
  Navier-Stokes equations \eqref{NS}
    are persistence of energy  if one of the following four conditions is satisfied
  \begin{enumerate}[(1)]
 \item $v\in L^{\f{2p}{p-1}}(0,T;\,L^{\f{2q_{1}}{q_{1}-1}}L^{\f{2q_{2}}{q_{2}-1}}
L^{\f{2q_{3}}{q_{3}-1}}
(\mathbb{R}^{3}))~\text{and}~ \nabla v \in L^{p} (0, T ; L^{\overrightarrow{q}} (\mathbb{R}^{3} ) ),~1<p,\overrightarrow{q}\leq\infty$;
 \item  $v\in L^{p}(0,T;L^{\overrightarrow{q}}(\mathbb{R}^{3})),~\text{with}~\f{1}{p}+
 \f{1}{q_{1}}+
 \f{1}{q_{2}}+
 \f{1}{q_{3}}=1~\text{and}~ \sum_{i=1}^{3}\frac{1}{q_{i}}\leq1,~1<\overrightarrow{q}\leq4$;
 \item $\nabla v \in L^{p} (0, T ; L^{\overrightarrow{q}} (\mathbb{R}^{3}) ),~\text{with}~
\frac  {1}{p}+\f{1}{q_{1}}+
 \f{1}{q_{2}}+
 \f{1}{q_{3}} =2~\text{and}~ \sum_{i=1}^{3}\frac{1}{q_{i}}\leq2$, $~1<\overrightarrow{q}\leq\f95$;

 \item $\nabla v \in L^{p} (0, T ; L^{\overrightarrow{q}} (\mathbb{R}^{3}) ),~\text{with}~
\frac{1}{p}+\frac{2}{5  }(\f{1}{q_{1}}+
 \f{1}{q_{2}}+
 \f{1}{q_{3}})=1~\text{and}~\sum_{i=1}^{3}\frac{1}{q_{i}}\leq\f53,
 ~1<\overrightarrow{q}\leq3.$

 \end{enumerate}
 \end{theorem}
 \begin{remark}
The first criterion in Theorem \ref{the1.3}  is  a generalization of all the corresponding known results in terms of the velocity and the gradient of the velocity.
 For the Dirichlet problem on a bounded smooth domain $\Omega$,
 one may establish similar results to this theorem
 by   the  approximation sequence developed   in \cite{[Galdi],[Shinbrot],[Lions],[BC],[WMH]}.
 \end{remark}

 \begin{remark}
 In the spirit of \cite{[WMH],[CL],[CCFS]}, we invoke the Littlewood-Paley theory in the proof of (1) in Theorem \ref{the1.3} to show that energy conservation class \eqref{wy} can also be extended in context of anisotropic Lebesgue spaces.
 The proof of (2) in this theorem is  a combination of  Gagliardo-Nirenberg inequality \eqref{glxied1} and Lions's class $v\in L^{4}(0,T;L^{4}(\mathbb{R}^{3})).$ We conclude the results in
 (3) and (4) by (1)  together with the Gagliardo-Nirenberg inequality \eqref{glxied1}.
 \end{remark}
 \begin{remark}
 The sufficient class (1) for  the weak solutions keeping the energy covers \eqref{wy} on the wholes space in \cite{[WMH]}.
The criteria (2) and (3) in  Theorem \ref{the1.3}  are a  generalization of  corresponding results in \cite{[BY],[Zhang]}, where the energy of 3D Navier-Stokes equations \eqref{NS} is conserved provided
$v\in L^{p}(0,T;L^{q}(\mathbb{R}^{3}))$ with $\f{1}{p}+
 \f{3}{q}=1$ and $3<q< 4,$ or $\nabla v \in L^{p} (0, T ; L^{q} (\mathbb{R}^{3} ) )$ with $\frac{1}{p}+\frac{3 }{q}=2$ and $\frac{3 }{2}<q<\frac{9}{5}$.
 \end{remark}

The rest of this  paper is organized as follows. In Section 2,
we  recall some basic  materials of  anisotropic Lebesgue spaces and anisotropic Besov spaces. Several key lemmas such as Bernstein type inequalities in anisotropic Lebesgue spaces and an approximate identity lemma are also established in this section.
Section 3 is devoted to the proof of Theorem \ref{the1.1}.
Section 4 contains the proof of Gagliardo-Nirenberg inequality \eqref{glibxied} in anisotropic Besov
  spaces. In Section 5, we present an application  of Gagliardo-Nirenberg inequality \eqref{glxied1} in anisotropic Lebesgue spaces to the
   energy conservation criteria of weak solutions to 3D Navier-Stokes equations, and finally prove Theorem \ref{the1.3}.
\section{Notations and  key auxiliary lemmas} \label{section2}
To begin with, we introduce some notations used in this paper.
 For $p\in [1,\,\infty]$, the notation $L^{p}(0,T;X)$ stands for the set of measurable functions $f(x,t)$ on the interval $(0,T)$  with values in $X$ and $\|f(\cdot,t)\|_{X}$ belonging to $L^{p}(0,T)$.
For a vector $\overrightarrow{p}=(p_1,p_2,\ldots,p_n)\in [1,\infty]^{n},$ we denote $1/\overrightarrow{p}=(1/p_1,1/p_2,\ldots,1/p_n).$ An inequality  between the vector $\overrightarrow{p}$ and a real number $\beta\in [1,\infty]$ is understood as the corresponding  inequality between each component $p_j$ and the number $\beta$ for $j=1,2,\ldots,n$. An  inequality  between  vectors  in  $\mathbb{R}^{n}$ signifies  the  corresponding  inequality  between  all  the  respective components. For simplicity, we write $\sum\limits_{i} f_{i}=: \sum f_{i}.$ The sign function $sgn(x)$ is defined by
$$
sgn(x)=\left\{\ba
-1& \quad\text { if }~ x<0, \\
0& \quad\text { if }~ x=0,\\
1& \quad\text { if }~ x>0.
\ea\right.
$$
Throughout this paper, $C$ is an absolute constant which may be different from line to line unless otherwise stated. $a\approx b$ means that $C^{-1}b\leq a\leq Cb$ for some constant $C>1$. $\chi_{\Omega}$ stands for the characteristic function of the set $\Omega\subset \mathbb{R}^{n}$. $|E|$ represents the $n$-dimensional Lebesgue measure of a set $E\subset \mathbb{R}^{n}$. Let $\mathcal{M}$ be the Hardy-Littlewood maximal operator and its definition is given by
  $$\mathcal{M}f(x)=\sup_{r>0}\f{1}{|B(r)|}\int_{B(r)}|f(x-y)|dy,$$
where $f$ is any locally integrable function on $\mathbb{R}^{n}$, and $B(r)$ is the open ball centered at the origin with radius $r>0$.

\subsection{Basic materials on anisotropic Lebesgue spaces}

For $x=(x_1,x_2,\ldots,x_n)\in\mathbb{R}^{n}$ and $\overrightarrow{q}=(q_1,q_2,\ldots,q_n)\in [1,\infty]^{n}$,
a function $f(x)$ belongs to  the anisotropic Lebesgue space $L^{\overrightarrow{q}}(\mathbb{R}^{n})$
if
  $$\|f\|_{L^{\overrightarrow{q}}(\mathbb{R}^{n})}=\|f\|_{L^{q_{1}}L^{q_{2}}\cdots L^{q_{n}}(\mathbb{R}^{n})}=
  \B\|\cdots\big\|\|f\|_{L_{x_1}^{q_{1}}(\mathbb{R})}\big\|_{L_{x_2}^{q_{2}}(\mathbb{R})}\cdots\B\|_{L_{x_n}^{q_{n}} (\mathbb{R})}<\infty.  $$
The study of anisotropic Lebesgue spaces  first appears in Benedek-Panzone \cite{[BP]}.
We list several useful properties of anisotropic Lebesgue spaces as follows.
\begin{itemize}
\item The H\"older inequality in anisotropic Lebesgue spaces \cite{[BP]}
 \be\label{HIAL}
 \|fg\|_{L^{1}(\mathbb{R}^{n})}\leq \|f \|_{L^{\overrightarrow{r}}(\mathbb{R}^{n})} \| g\|_{L^{\overrightarrow{s}}(\mathbb{R}^{n})}~~\text{with}~~ \f{1}{\overrightarrow{r}}+\f{1}{\overrightarrow{s}}=1,~1\leq\overrightarrow{r}\leq\infty.
 \ee

\item The interpolation  inequality in anisotropic Lebesgue spaces \cite{[BP]}
        \be\label{interi}
\|f\|_{L^{\overrightarrow{r}}(\mathbb{R}^{n}) }
\leq \|f\|^{\alpha}_{L^{\overrightarrow{s}}(\mathbb{R}^{n}) }\|f\|^{1-\alpha}_{L^{\overrightarrow{t}}(\mathbb{R}^{n}) }~~\text{with}~~\f{1}{\overrightarrow{r}}=\f{\alpha}{\overrightarrow{s}}+\f{1-\alpha}{\overrightarrow{t}},~0\leq\alpha\leq 1,~1\leq\overrightarrow{s},\overrightarrow{t}\leq\infty.
\ee
\item  Young's  inequality in anisotropic Lebesgue spaces  \cite{[BP]} \\
\begin{align}\label{YoungI}
\|f\ast g\|_{L^{\overrightarrow{r}}(\mathbb{R}^{n}) }
\leq \|f\|_{L^{\overrightarrow{p}}(\mathbb{R}^{n}) }\|g\|_{L^{\overrightarrow{q}}(\mathbb{R}^{n})}~~\text{with}~~1+\frac{1}{\overrightarrow{r}}=\frac{1}{\overrightarrow{p}}+\frac{1}{\overrightarrow{q}},~1\leq \overrightarrow{p},\overrightarrow{q}\leq \infty.
\end{align}
\end{itemize}

\subsection{Generalized Bernstein inequality in anisotropic Lebesgue spaces}
As an application of Young's inequality \eqref{YoungI}, we may derive the generalized Bernstein inequalities in anisotropic Lebesgue spaces as follows.
\begin{lemma}\label{lem2.1}
Let a ball $B=\left\{\xi \in \mathbb{R}^{n}: |\xi| \leq R\right\}$ with $0<R<\infty$ and an annulus $\mathcal{C}=\left\{\xi \in \mathbb{R}^{n}: r_{1} \leq|\xi| \leq r_{2}\right\}$ with $0<r_{1}<r_{2}<\infty$. Then a constant $C>1$ exists such that for any nonnegative integer $k,$ any couple $(\overrightarrow{p}, \overrightarrow{q})$ with $1\leq\overrightarrow{p} \leq\overrightarrow{q}\leq\infty,$ any $\lambda\in (0, \infty)$, and any function $u$ in $L^{\overrightarrow{p}}(\mathbb{R}^{n})$, there hold
\begin{align}
 &  \sup _{|\alpha|=k}\left\|\partial^{\alpha} u\right\|_{L^{\overrightarrow{q} }(\mathbb{R}^{n})} \leq C \lambda^{k+\sum_{i}\left(\frac{1}{p_{i}}-\frac{1}{q_{i}}\right)}\|u\|_{L^{\overrightarrow{p} }(\mathbb{R}^{n})} ~~ \text{with}~~ \operatorname{supp} \widehat{u} \subset \lambda B,\label{bern2}\\
&  C^{-1} \lambda^{k}\|u\|_{L^{\overrightarrow{p}}(\mathbb{R}^{n})} \leq \sup _{|\alpha|=k}\left\|\partial^{\alpha} u\right\|_{L^{\overrightarrow{p}}(\mathbb{R}^{n})} \leq C \lambda^{k}\|u\|_{L^{\overrightarrow{p}}(\mathbb{R}^{n})} ~~ \text{with}~~ \operatorname{supp} \widehat{u} \subset \lambda \mathcal{C},\label{bern4}
\end{align}
where $\lambda B=\left\{\xi \in \mathbb{R}^{n}: |\xi| \leq \lambda R\right\}$ and $\lambda \mathcal{C}=\left\{\xi \in \mathbb{R}^{n}: \lambda r_{1} \leq|\xi| \leq \lambda r_{2}\right\}$.
\end{lemma}
\begin{proof}
(1)\,\,Let $\psi$ be a Schwartz function on $\mathbb{R}^{n}$ such that $\chi_{B}\leq \hat{\psi}\leq \chi_{2B}$. Since $\widehat{u}(\xi)=\hat{\psi}(\xi/\lambda) \widehat{u}(\xi)$ when $\operatorname{supp} \widehat{u} \subset \lambda B$, we have
$$
\partial^{\alpha} u=i^{|\alpha|}\mathcal{F}^{-1}(\xi^{\alpha}\hat{u})=
i^{|\alpha|}\mathcal{F}^{-1}(\xi^{\alpha}\hat{\psi}(\xi/\lambda) \widehat{u}(\xi))=
\lambda^{|\alpha|}(\partial^{\alpha} \psi)_{\lambda}\ast u.
$$
Here  $(\partial^{\alpha} \psi)_{\lambda}(x)=\lambda^{n}\,\partial^{\alpha} \psi(\lambda x)$ for all $x\in\mathbb{R}^{n}$.

Fix $|\alpha|=k$. Take $1/\overrightarrow{r}=1+1/\overrightarrow{q}-1/\overrightarrow{p}$, then the hypotheses on the indices imply that $1\leq\overrightarrow{r}\leq\infty$. In view of \eqref{YoungI}, we infer that
$$\ba
\|\partial^{\alpha}u \|_{L^{\overrightarrow{q}}(\mathbb{R}^{n})}=& \lambda^{k}\left\|(\partial^{\alpha} \psi)_{\lambda}\ast u\right\|_{L^{\overrightarrow{q}}(\mathbb{R}^{n})}\\
\leq & \lambda^{k}\left\|(\partial^{\alpha} \psi)_{\lambda}\right\|_{L^{\overrightarrow{r}}(\mathbb{R}^{n})}\|u\|_{L^{\overrightarrow{p}}(\mathbb{R}^{n})}\\
=& \lambda^{k+n-\sum_{i}\frac{1}{r_{i}}}\left\|\partial^{\alpha} \psi\right\|_{L^{\overrightarrow{r}}(\mathbb{R}^{n})}\|u\|_{L^{\overrightarrow{p}}(\mathbb{R}^{n})}\\
=& \lambda^{k+\sum_{i}\left(\frac{1}{p_{i}}-\frac{1}{q_{i}}\right)}\left\|\partial^{\alpha} \psi\right\|_{L^{\overrightarrow{r}}(\mathbb{R}^{n})}\|u\|_{L^{\overrightarrow{p}}(\mathbb{R}^{n})}.
\ea$$
This together with the fact that $\partial^{\alpha} \psi \in \mathcal{S}(\mathbb{R}^{n})\subset L^{\overrightarrow{r}}(\mathbb{R}^{n})$ yields \eqref{bern2}.

(2) \,\,Observe that \eqref{bern2} with $\overrightarrow{q}=\overrightarrow{p}$ implies the second inequality of \eqref{bern4}, it suffices to show the first inequality in \eqref{bern4}. Let $\eta$ be a Schwartz function on $\mathbb{R}^{n}$ such that $\chi_{\mathcal{C}}\leq \hat{\eta}\leq \chi_{\tilde{\mathcal{C}}}$, where $\tilde{\mathcal{C}}=\left\{\xi \in \mathbb{R}^{n}:  r_{1}/2 \leq|\xi| \leq 2 r_{2}\right\}$. It follows from $ \operatorname{supp} \hat{u} \subset\lambda \mathcal{C}$ that for all $\xi\in\mathbb{R}^{n}$,
$$
\hat{u}(\xi)=\sum_{|\alpha|=k} \frac{(-i \xi)^{\alpha}}{|\xi|^{2 k}} \hat{\eta}\left(\xi/\lambda\right)(i \xi)^{\alpha} \hat{u}(\xi)=\lambda^{-k} \sum_{|\alpha|=k} \frac{(-i \xi / \lambda)^{\alpha}}{|\xi/ \lambda|^{2 k}} \hat{\eta}\left(\xi/\lambda\right)\mathcal{F}(\partial^{\alpha}u)(\xi).  $$
Therefore, we may write
$$
u=\lambda^{-k} \sum_{|\alpha|=k} (g_{\alpha})_{\lambda} \ast \partial^{\alpha} u,
$$
where $(g_{\alpha})_{\lambda}(x)=\lambda^{n} g_{\alpha}(\lambda x)$ for all $x\in\mathbb{R}^{n}$, and
$$
 g_{\alpha}=\mathcal{F}^{-1}\left(\frac{(-i \xi)^{\alpha}}{|\xi|^{2 k}} \hat{\eta}(\xi)\right) \in \mathcal{S}(\mathbb{R}^{n})\subset L^{1}(\mathbb{R}^{n}).
$$
This together with Young's inequality \eqref{YoungI} yields that
$$
\begin{aligned}
\|u\|_{L^{\overrightarrow{p}}(\mathbb{R}^{n})} & \leq \lambda^{-k}\sum_{|\alpha|=k}\left\|(g_{\alpha})_{\lambda} \ast \partial^{\alpha} u\right\|_{L^{\overrightarrow{p}}(\mathbb{R}^{n})} \\
& \leq \lambda^{-k}\sum_{|\alpha|=k}\left\|(g_{\alpha})_{\lambda}\right\|_{L^{1}(\mathbb{R}^{n})}\left\|\partial^{\alpha} u\right\|_{L^{\overrightarrow{p}}(\mathbb{R}^{n})} \\
 &\leq \lambda^{-k}\left(\sum_{|\alpha|=k}\left\|g_{\alpha}\right\|_{L^{1}(\mathbb{R}^{n})}\right) \sup _{|\alpha|=k}\left\|\partial^{\alpha} u\right\|_{L^{\overrightarrow{p}}(\mathbb{R}^{n})}.
\end{aligned}
$$
This concludes \eqref{bern4} and the proof is completed.
\end{proof}
\subsection{Anisotropic Besov spaces and anisotropic Sobolev spaces}
  $\mathcal{S}$ represents the Schwartz class of rapidly decreasing functions, $\mathcal{S}'$ the
space of tempered distributions, $\mathcal{S}'/\mathcal{P}$ the quotient space of tempered distributions which modulo polynomials.
  We use $\mathcal{F}f$ or $\widehat{f}$ to denote the Fourier transform of a tempered distribution $f$.
To define anisotropic Besov space, we need the following dyadic unity partition
(see e.g. \cite{[BCD]}). Choose two nonnegative radial
functions $\varrho$, $\varphi\in C^{\infty}(\mathbb{R}^{n})$ be
supported respectively in the ball $\{\xi\in
\mathbb{R}^{n}:|\xi|\leq \frac{4}{3} \}$ and the shell $\{\xi\in
\mathbb{R}^{n}: \frac{3}{4}\leq |\xi|\leq
  \frac{8}{3} \}$ such that
\begin{equation*}
\varrho(\xi)+\sum_{j\geq 0}\varphi(2^{-j}\xi)=1, \quad
 \forall\xi\in\mathbb{R}^{n}; \qquad
 \sum_{j\in \mathbb{Z}}\varphi(2^{-j}\xi)=1, \quad \forall\xi\neq 0.
\end{equation*}
The nonhomogeneous dyadic
blocks $\Delta_{j}$ are defined by
$$\left\{\ba
  &  \Delta_{j} f:=0 ~~ \text{if} ~~ j \leq-2,\\
&\Delta_{-1} f:=\varrho(D) f=(\varrho(\xi)\hat{f}(\xi))^{\vee},\\
&\Delta_{j} f:=\varphi\left(2^{-j} D\right) f=2^{jn}\check{\varphi}(2^j\cdot)\ast f  ~~\text{if}~~ j \geq 0,\\
&S_{j}f:= \sum_{k\leq j-1}\Delta_{k}f ~~\text{for}~~j\in\mathbb{Z}.
\ea\right.$$
The homogeneous Littlewood-Paley operators are defined as follows
\begin{equation*}
  \forall j\in\mathbb{Z},\quad \dot{\Delta}_{j}f:= \varphi(2^{-j}D)f \quad\text{and}\quad \dot S_{j}f:= \sum_{k\leq j-1}\dot\Delta_{k}f.
\end{equation*}
The homogeneous anisotropic Besov space $ \dot{B}^{s}_{\overrightarrow{p},r}(\mathbb{R}^{n})$, with $s\in \mathbb{R}$, $r\in [1,\infty]$ and $\overrightarrow{p}\in [1,\infty]^{n},$ is the set of $f\in\mathcal{S}'(\mathbb{R}^{n})/\mathcal{P}(\mathbb{R}^{n})$ such that
\begin{equation*}
  \norm{f}_{\dot{B}^{s}_{\overrightarrow{p},r}(\mathbb{R}^{n})}:=\norm{\left\{2^{js}\norm{\dot{\Delta}
  _{j}f}_{L^{\overrightarrow{p}}(\mathbb{R}^{n})}\right\}}_{\ell^{r}(\mathbb{Z})}<\infty,
\end{equation*}
where $\ell^{r}(\mathbb{Z})$ represents the set of sequences with summable $r$-th powers.
The homogenous anisotropic Sobolev norm $\|\cdot\|_{\dot{H}^{s}_{\overrightarrow{p} }(\mathbb{R}^{n})}$  is defined as $\|f\|_{\dot{H}^{s}_{\overrightarrow{p}}(\mathbb{R}^{n})}=\|\Lambda^{s}f\|_{L^{\overrightarrow{p}}(\mathbb{R}^{n})}.$
When $p_{1}=\cdots=p_{n}=p$, the anisotropic Sobolev spaces reduce to the classical Sobolev spaces $\dot{H}^{s}_{p}(\mathbb{R}^{n})$.

\begin{lemma}\label{lem2.2}
Suppose that $f\in\dot{H}^{s}_{\overrightarrow{p}}(\mathbb{R}^{n})$ with $s\in \mathbb{R}$ and $\overrightarrow{p}\in [1,\infty]^{n}$. Then there holds
\be\label{keyinclue}
 \|f\|_{ \dot{B}^{s}_{\overrightarrow{p}, \infty }(\mathbb{R}^{n})}\leq C\|f\|_{\dot{H}^{s}_{\overrightarrow{p}}(\mathbb{R}^{n})},
\ee
where $C>0$ depends only on $n, s$ and $\varphi$.
\end{lemma}
\begin{proof}
It suffices to show that for all functions $g\in L^{\overrightarrow{p}}(\mathbb{R}^{n})$,
$$\sup _{j \in \mathbb{Z}}\| \rho_{j}\ast g  \|_{L^{\overrightarrow{p}}(\mathbb{R}^{n})}\leq C\|g\|_{L^{\overrightarrow{p}}(\mathbb{R}^{n})},$$
where $\rho=\left(|\xi|^{-s}\varphi(\xi)\right)^{\vee}$ and $\rho_{j}(x)=2^{jn}\rho(2^{j}x)$ for all $x\in\mathbb{R}^{n}$.

To this end, observe that $\rho$ is a Schwartz function on $\mathbb{R}^{n}$ and $\|\rho\|_{L^{1}(\mathbb{R}^{n})}<\infty.$ This together with Young's inequality \eqref{YoungI} implies that
\begin{align*}
\sup _{j \in \mathbb{Z}}\| \rho_{j}\ast g  \|_{L^{\overrightarrow{p}}(\mathbb{R}^{n})}\leq \sup _{j \in \mathbb{Z}} \|\rho_{j}\|_{L^{1}(\mathbb{R}^{n})}
\|g  \|_{L^{\overrightarrow{p}}(\mathbb{R}^{n})}=\|\rho\|_{L^{1}(\mathbb{R}^{n})}\|g \|_{L^{\overrightarrow{p}}(\mathbb{R}^{n})},
\end{align*}
which concludes the proof.
\end{proof}

We end this section with an anisotropic version of approximate identity lemma, which will be used in the proof of Theorem \ref{the1.3}.
\begin{lemma}\label{lem2.3}
Assume that $f(x,t)\in L^{p} (0, T ; L^{\overrightarrow{q}} (\mathbb{R}^{n} ) )$ with $p\in [1,\infty)$ and $\overrightarrow{q}\in [1,\infty)^{n}$. Let $S_{j}f(x,t)=\varrho(2^{-j}D) f(x,t)=\mathcal{F}^{-1}_{\xi\rightarrow x}(\varrho(2^{-j}\xi)\,\mathcal{F}_{y\rightarrow\xi}(f(y,t)))$.
Then there holds
$$\big\|S_{j}f -f\big\|_{L^{p} (0, T ; L^{\overrightarrow{q}} (\mathbb{R}^{n} ))}\rightarrow0,~\text{as}~j\rightarrow\infty.
$$
\end{lemma}
\begin{proof}
Note that $\int_{\mathbb{R}^{n}}\check{\varrho}=\varrho(0)=1$, we have
\begin{align*}
S_{j}f(x,t) -f(x,t)=&\int_{\mathbb{R}^{n}} \left(f(x-y,t)-f(x,t)\right)\check{\varrho}_{j}(y)dy\\
=&\int_{\mathbb{R}^{n}} \left(f(x-2^{-j}y,t)-f(x,t)\right)\check{\varrho}(y)dy
\end{align*}
for any $x\in\mathbb{R}^{n}$ and $t\in (0, T)$, where $\check{\varrho}_{j}(y)=2^{jn}\check{\varrho}(2^{j}y)$ for all $y\in\mathbb{R}^{n}$. Then the generalized Minkowski inequality in \cite[Theorem 2.10]{[BIN]} enables us to arrive at
\begin{align*}
\big\|S_{j}f(\cdot,t) -f(\cdot,t)\big\|_{L^{\overrightarrow{q}} (\mathbb{R}^{n} )}\leq \int_{\mathbb{R}^{n}} \big\|f(\cdot-2^{-j}y,t)-f(\cdot,t)\big\|_{L^{\overrightarrow{q}} (\mathbb{R}^{n} )} |\check{\varrho}(y)|dy.
\end{align*}
Since the translation operator is strongly continuous on $L^{\overrightarrow{q}} (\mathbb{R}^{n} )$ with $\overrightarrow{q}\in [1,\infty)^{n}$ (see \cite[Theorem 1.5]{[BIN]}), there holds for all $y\in\mathbb{R}^{n}$ and almost every $t\in (0, T)$,
$$\lim_{j\rightarrow\infty}\big\|f(\cdot-2^{-j}y,t)-f(\cdot,t)\big\|_{L^{\overrightarrow{q}} (\mathbb{R}^{n} )}=0.$$
Combining this with the fact that $\big\|f(\cdot-2^{-j}y,t)-f(\cdot,t)\big\|_{L^{\overrightarrow{q}} (\mathbb{R}^{n} )}\leq 2\big\|f(\cdot,t)\big\|_{L^{\overrightarrow{q}} (\mathbb{R}^{n} )}$ and $\check{\varrho}\in \mathcal{S}(\mathbb{R}^{n})\subset L^{1}(\mathbb{R}^{n})$, we may derive from Lebesgue's Dominated Convergence Theorem that for almost every $t\in (0, T)$,
$$\limsup_{j\rightarrow\infty}\big\|S_{j}f(\cdot,t) -f(\cdot,t)\big\|_{L^{\overrightarrow{q}} (\mathbb{R}^{n} )}\leq \lim_{j\rightarrow\infty}\int_{\mathbb{R}^{n}} \big\|f(\cdot-2^{-j}y,t)-f(\cdot,t)\big\|_{L^{\overrightarrow{q}} (\mathbb{R}^{n} )} |\check{\varrho}(y)|dy=0.$$
This together with the observation that
$$\big\|S_{j}f(\cdot,t) -f(\cdot,t)\big\|_{L^{\overrightarrow{q}} (\mathbb{R}^{n} )}\leq 2\big\|f(\cdot,t)\big\|_{L^{\overrightarrow{q}} (\mathbb{R}^{n} )}\int_{\mathbb{R}^{n}} |\check{\varrho}(y)|dy \in L^{p}(0, T)$$
yields that
$$  \lim_{j\rightarrow\infty}\int_{0}^{T}\big\|S_{j}f(\cdot,t) -f(\cdot,t)\big\|^{p}_{L^{\overrightarrow{q}} (\mathbb{R}^{n} )}dt=0.$$
Here we have used the Lebesgue's Dominated Convergence Theorem once again, which concludes the proof of this lemma.
\end{proof}
\section{Gagliardo-Nirenberg inequality in mixed Lebesgue spaces }
This section is    devoted to    the proof of Theorem \ref{the1.1}, which establishes Gagliardo-Nirenberg inequality \eqref{glxied1} in anisotropic Lebesgue spaces.
\begin{proof}[Proof of Theorem \ref{the1.1}]
Since the estimate \eqref{glxied1} with $n=1$ is the well-known fractional Gagliardo-Nirenberg inequality in isotropic Lebesgue spaces on $\mathbb{R}$, we only need to consider the case that $n\geq 2.$

By applying the celebrated Marcinkiewicz-Lizorkin multiplier theorem in \cite{[Lizork]} to the Fourier multiplier operator on $\mathbb{R}^{n}$ with each of the following three symbols
$$
m_{1}(\xi)=\frac{|\xi_{i}|^{s}}{\sum\limits_{i=1}^{n}|\xi_{i}|^{s}}~~\text{with}~~1\leq i\leq n,
~~m_{2}(\xi)=\frac{\sum\limits_{i=1}^{n}|\xi_{i}|^{s}}{|\xi|^{s}},
~~m_{3}(\xi)=\frac{1}{m_{2}(\xi)},
$$
we arrive at
$$
\|\Lambda^{s}u\|_{L^{\overrightarrow{p} } (\mathbb{R}^{n})}\approx \|\sum\limits_{i=1}^{n}\Lambda_{i}^{s}u\|_{L^{\overrightarrow{p} } (\mathbb{R}^{n})}
\approx \sum\limits_{i=1}^{n}\|\Lambda_{i}^{s}u\|_{L^{\overrightarrow{p} } (\mathbb{R}^{n})}
$$
for any $s\geq 0$ and $1<\overrightarrow{p}<\infty$. Here $\Lambda=(-\Delta)^{1/2}$ and $\Lambda_{i}=(-\partial_{i}^{2})^{1/2}$ for $1\leq i\leq n$.\\
Hence, the desired inequality \eqref{glxied1} is equivalent to
$$
\sum\limits_{i=1}^{n}\|\Lambda_{i}^{\sigma}u\|_{L^{\overrightarrow{p} } (\mathbb{R}^{n})} \leq C\|u\|_{L^{\overrightarrow{q}}(\mathbb{R}^{n}) }^{\theta}\left(\sum\limits_{i=1}^{n}\|\Lambda_{i}^{s}u\|_{L^{\overrightarrow{r} } (\mathbb{R}^{n})}\right)^{1-\theta}.$$
Furthermore, this estimate together with the standard rescaling analysis implies the following equivalent inequality that for $k=1,2,\ldots,n$,
\be\label{scal}
\|\Lambda_{k}^{\sigma}u\|_{L^{\overrightarrow{p} } (\mathbb{R}^{n})} \leq C\|u\|_{L^{\overrightarrow{q}}(\mathbb{R}^{n}) }^{\theta}\left(\lambda_{k}^{-\sigma}\prod\limits_{j=1}^{n}\lambda_{j}^{\tau_{j}}\right)\left(\sum\limits_{i=1}^{n}\lambda_{i}\|\Lambda_{i}^{s}u\|^{\frac{1}{s}}_{L^{\overrightarrow{r} } (\mathbb{R}^{n})}\right)^{s(1-\theta)},
\ee
where $\tau_{j}=\frac{1}{p_{j}}-\frac{\theta}{q_{j}}-\frac{1-\theta}{r_{j}}$ for $1\leq j\leq n$. For $n\geq2$, this yields that $\sigma-s(1-\theta)\leq\tau_{j}\leq 0$ with $1\leq j\leq n$, otherwise we can derive a contradiction for any nonzero function $u$ by letting some positive $\lambda_{i}$ in \eqref{scal} tend to zero or infinity appropriately. Also, by taking $\lambda_{1}=\lambda_{2}=\cdots=\lambda_{n}>0,$ we infer that
$$
\sum\limits_{j=1}^{n}(-\tau_{j})=s(1-\theta)-\sigma\geq 0.
$$\
This means the neccessity of condition \eqref{condi} for anisotropic inequality \eqref{glxied1} in Theorem \ref{the1.1}.

In what follows, we shall give a proof of the Gagliardo-Nirenberg inequality \eqref{glxied1} under the condition \eqref{condi}.

  Since the case $\theta=1$ implies that $\sigma=0$ and $\overrightarrow{q}=\overrightarrow{p}$, it suffices to prove \eqref{glxied1} for the case $\theta<1$. The proof can be divided into two steps.

Firstly, we shall show
$$
\|\Lambda^{\sigma}u\|_{L^{\overrightarrow{p} } (\mathbb{R}^{n})} \leq C \left\|\Lambda^{s(1-\theta)}u\right\|_{L^{\overrightarrow{\beta} } (\mathbb{R}^{n})},
$$
or equivalently,
$$
\B\|(\prod\limits_{j=1}^{n}\Lambda^{\tau_{j}}) f\B\|_{L^{\overrightarrow{p} } (\mathbb{R}^{n})}=\B\|\Lambda^{\sum\limits_{j=1}^{n}\tau_{j}} f \B\|_{L^{\overrightarrow{p} } (\mathbb{R}^{n})}=\left\|\Lambda^{\sigma-s(1-\theta)} f\right\|_{L^{\overrightarrow{p} } (\mathbb{R}^{n})} \leq C \|f\|_{L^{\overrightarrow{\beta} } (\mathbb{R}^{n})},
$$
where
$$
\overrightarrow{\tau}=\frac{1}{\overrightarrow{p}}-\frac{1}{\overrightarrow{\beta}},\quad 1<\overrightarrow{\beta}\leq \overrightarrow{p}<\infty.
$$
Note that the celebrated Marcinkiewicz-Lizorkin multiplier theorem in \cite{[Lizork]} applied to the Fourier multiplier operator on $\mathbb{R}^{n}$ with the symbol $\frac{\prod\limits_{j=1}^{n}|\xi|^{\tau_{j}}}{\prod\limits_{j=1}^{n}|\xi_{j}|^{\tau_{j}}}$ yields the estimate
$$
\B\|(\prod\limits_{j=1}^{n}\Lambda^{\tau_{j}}) f\B\|_{L^{\overrightarrow{p} } (\mathbb{R}^{n})}\leq C\B\|(\prod\limits_{j=1}^{n}\Lambda_{j}^{\tau_{j}}) f\B\|_{L^{\overrightarrow{p} } (\mathbb{R}^{n})},
$$
we only need to prove
$$
\B\|(\prod\limits_{j=1}^{n}\Lambda_{j}^{\tau_{j}}) f\B\|_{L^{\overrightarrow{p} } (\mathbb{R}^{n})}\leq C\|f\|_{L^{\overrightarrow{\beta} } (\mathbb{R}^{n})}.
$$
Indeed, the classical Hardy-Littlewood-Sobolev inequality on fractional integration and the H\"{o}lder inequality enable us to deduce that
\begin{align*}
&\B|\int_{\mathbb{R}^{n}} g(y_{1},\ldots,y_{n})\cdot(\prod\limits_{j=1}^{n}\Lambda_{j}^{\tau_{j}})f(y_{1},\ldots,y_{n})\,dy_{1}\cdots dy_{n}\B|\\
=&\B|\int_{\mathbb{R}^{2n}} f(x_{1},\ldots,x_{n})g(y_{1},\ldots,y_{n})\\
&\times\prod\limits_{j=1}^{n}
\B\{|x_{j}-y_{j}|^{sgn(\tau_{j})(\tau_{j}+1)}\delta(x_{j}-y_{j})^{sgn(\tau_{j})+1}B\} dx_{1}\cdots dx_{n}dy_{1}\cdots dy_{n}\B|\\
\leq & \int_{\mathbb{R}^{2n-2}}\prod\limits_{j=2}^{n}\left\{|x_{j}-y_{j}|^{sgn(\tau_{j})(\tau_{j}+1)}\delta(x_{j}-y_{j})^{sgn(\tau_{j})+1}\right\}dx_{2}\cdots dx_{n}dy_{2}\cdots dy_{n} \\
&\times\int_{\mathbb{R}^{2}}|f(x_{1},\ldots,x_{n})|\cdot|g(y_{1},\ldots,y_{n})|\cdot|x_{1}-y_{1}|^{sgn(\tau_{1})(\tau_{1}+1)}\delta(x_{1}-y_{1})^{sgn(\tau_{1})+1}dx_{1}dy_{1}\\
\leq & C\int_{\mathbb{R}^{2n-2}}\|f(\,\cdot,x_{2},\ldots,x_{n})\|_{L_{1}^{\beta_{1}}(\mathbb{R})}\|g(\,\cdot,y_{2},\ldots,y_{n})\|_{L_{1}^{p_{1}'}(\mathbb{R})}\\
& \times\prod\limits_{j=2}^{n}\left\{|x_{j}-y_{j}|^{sgn(\tau_{j})(\tau_{j}+1)}\delta(x_{j}-y_{j})^{sgn(\tau_{j})+1}\right\}dx_{2}\cdots dx_{n}dy_{2}\cdots dy_{n},
\end{align*}
where $\frac{1}{\overrightarrow{p}'}=1-\frac{1}{\overrightarrow{p}}\,$ and $\delta$ represents the Dirac delta function at the origin. By induction, it can easily be seen that
$$
\left|\int_{\mathbb{R}^{n}} g(y)\cdot(\prod\limits_{j=1}^{n}\Lambda_{j}^{\tau_{j}})f(y)dy\right|\leq C\|f\|_{L^{\overrightarrow{\beta}}(\mathbb{R}^{n})} \|g\|_{L^{\overrightarrow{p}'}(\mathbb{R}^{n})},
$$
which verifies the desired estimate.

Next, thanks to the decomposition of low and high frequencies, one can derive the following pointwise estimate
\be\label{bcd}
 \left|\Lambda^{s(1-\theta)} u(x)\right| \leq C(\mathcal{M} u(x))^{\theta}\left(\mathcal{M} \Lambda^{s} u(x)\right) ^{1-\theta},
\ee
whose proof can be found in \cite[p.84]{[BCD]}. As a consequence, it follows from the H\"older inequality that
$$\ba
 \left\|\Lambda^{s(1-\theta)}u\right\|_{L^{\overrightarrow{\beta} } (\mathbb{R}^{n})}&\leq C \|(\mathcal{M} u )^{\theta}\left(\mathcal{M} \Lambda^{s} u \right) ^{1-\theta}\|_{L^{\overrightarrow{\beta} } (\mathbb{R}^{n})}\\& \leq C \|\mathcal{M} u \|^{\theta}_{L^{\overrightarrow{q} }(\mathbb{R}^{n})}\|\mathcal{M}\left(  \Lambda^{s } u\right)\|^{1-\theta}_{L^{\overrightarrow{r}}(\mathbb{R}^{n})},
 \ea$$
where $\frac{1}{\overrightarrow{\beta}}=\frac{\theta}{\overrightarrow{q}}+\frac{1-\theta}{\overrightarrow{r}}$. According to the boundedness of Hardy-Littlewood maximal operator in mixed Lebesgue spaces (see \cite[Theorem 6.9]{[Sawano]}), we conclude the desired inequality \eqref{glxied1}.
The proof of Theorem \ref{the1.1} is completed.
\end{proof}

\section{Gagliardo-Nirenberg inequality in mixed Besov spaces}

The goal of this section is to prove Theorem   \ref{the1.2}  by following  the path of \cite{[Chikami],[WWY]}  via
the  generalized Bernstein inequalities in  Lemma \ref{lem2.1}.
 The proof  in \cite{[Chikami],[WWY]}   relies on the
relationship between $ \sum1/q_{i}  $ and $\sum1/r_{i}$.
Here, the different starting point in our proof of Theorem   \ref{the1.2}  is  based on the relationship between $ \sum1/q_{i}+\sigma $ and $\sum1/p_{i}$. Indeed, we observe that there holds
 $$(1-\theta)( {s-\sigma} +\sum\f{1}{p_{i}}-\sum\f{ 1}{r_{i}})=\theta(\sum\f1q_{i}+ \sigma  -\sum\f1p_{i}),$$
which enables us to divide the proof of Theorem \ref{the1.2} into three main cases.
  \begin{proof}[Proof of Theorem \ref{the1.2}]
{\bf Case 1:}  Firstly, we consider the case that $\sum\f1q_{i}+ {\sigma} >\sum\f1p_{i}$ , that is,  ${s-\sigma} +\sum\f{1}{p_{i}}-\sum\f{ 1}{r_{i}}>0$.
Our proof will be divided into three subcases by the relationship
between $\sum\f1r_{i}$ and $\sum\f1p_{i}$.

($I_{1}$) $\sum\f{1}{r_{i}}<\sum\f{1}{p_{i}}$:

($I_{11}$) $\sum\f{1}{r_{i}}<\sum\f{1}{p_{i}}=\sum\f{1}{q_{i}}$:

 Note that
\be\ba\label{mixed3.1}
\sum\f{1}{p_{i}}&=(1-\f{\sigma}{s})\sum\f{1}{q_{i}}+\sum\f{\sigma}{s}\f{1}{r_{i}}+
(1-\f{\sigma}{s})(\sum\f{1}{r_{i}}-\sum\f{1}{q_{i}})+(\sigma-s(1-\theta))\\
&< (1-\f{\sigma}{s})\sum\f{1}{q_{i}}+\f{\sigma}{s}\sum\f{1}{r_{i}}\\
&< (1-\f{\sigma}{s})\sum\f{1}{q_{i}}+\f{\sigma}{s}\sum\f{1}{q_{i}}
\\&=\sum\f{1}{q_{i}}.
\ea\ee
Hence, this subcase will not happen.

($I_{12}$) $ \sum\f1r_{i}<\sum\f1p_{i}< \sum\f1q_{i}$:

Owing to \eqref{addc} and \eqref{interi}, there holds
$$
\|\dot{\Delta}_{j}u\|_{L^{\overrightarrow{p}}(\mathbb{R}^{n})}\leq C
\|\dot{\Delta}_{j}u\|^{\alpha}_{L^{\overrightarrow{q} }(\mathbb{R}^{n})}
\|\dot{\Delta}_{j}u\|^{1-\alpha}_{L^{\overrightarrow{r} }(\mathbb{R}^{n})},
~~\f{1}{p_{i}}=\f{\alpha}{q_{i}}+\f{1-\alpha}{r_{i}}.
$$
Since $\sum\f1r_{i}<\sum\f1q_{i}$, we derive from \eqref{mixed3.1} that $\sum\f{1}{p_{i}} < (1-\f{\sigma}{s})\sum\f{1}{q_{i}}+\f{\sigma}{s}\sum\f{1}{r_{i}}$, which implies that
$s(1-\alpha)-\sigma>0$.\\
An application of Bernstein inequalities in Lemma \ref{lem2.1} yields that
\be\label{3.17}\ba
\|u\|_{\dot{B}^{\sigma}_{\overrightarrow{p},1 }}&=\sum_{j\leq k}2^{j\sigma}\|\dot{\Delta}_{j}u\|_{L^{\overrightarrow{p} }(\mathbb{R}^{n})}+\sum_{j> k}2^{j\sigma}\|\dot{\Delta}_{j}u\|_{L^{\overrightarrow{p} }(\mathbb{R}^{n})}\\
&\leq\sum_{j\leq k}2^{j\sigma}\|\dot{\Delta}_{j}u\|_{L^{\overrightarrow{p} }(\mathbb{R}^{n})}+\sum_{j> k}2^{j\sigma}\|\dot{\Delta}_{j}u\|^{\alpha}_{L^{\overrightarrow{q} }(\mathbb{R}^{n})}
\|\dot{\Delta}_{j}u\|^{1-\alpha}_{L^{\overrightarrow{r} }(\mathbb{R}^{n})}\\
&\leq C\f{2^{k[\sigma+\sum(\f1q_{i}-\f1p_{i})]}}{2^{ [\sigma+\sum(\f1r_{i}-\f1p_{i})]}-1} \| u\|_{\dot{B}^{0}_{\overrightarrow{q},\infty}}+C \f{ 2^{-k[s(1-\alpha)-\sigma ]}}{1-2^{- [s(1-\alpha)-\sigma ]}} \| u\|_{\dot{B}^{0}_{\overrightarrow{q},\infty}}^{\alpha} \| u\|_{\dot{B}^{s}_{\overrightarrow{r}, \infty}}^{1-\alpha},
\ea\ee
where we have used the fact that $\sigma+\sum(\f1q_{i}-\f1p_{i})>0$ and $s(1-\alpha)-\sigma>0$.\\
By choosing $k\in \mathbb{Z}$ such that
$$\f{2^{k[\sigma+\sum(\f1q_{i}-\f1p_{i})]}}{2^{ [\sigma+\sum(\f1r_{i}-\f1p_{i})]}-1} \| u\|_{\dot{B}^{0}_{\overrightarrow{q},\infty}}\approx\f{ 2^{-k[s(1-\alpha)-\sigma ]}}{1-2^{- [s(1-\alpha)-\sigma ]}}  \| u\|_{\dot{B}^{0}_{\overrightarrow{q},\infty}}^{\alpha} \| u\|_{\dot{B}^{s}_{\overrightarrow{r}, \infty}}^{1-\alpha},$$
we conclude that
$$\ba
\|u\|_{\dot{B}^{\sigma}_{\overrightarrow{p}, 1}} &\leq C  \| u\|_{\dot{B}^{0}_{\overrightarrow{r},\infty}} ^{\theta} \| u\|_{\dot{B}^{s}_{\overrightarrow{r}, \infty}}^{1-\theta}.
\ea$$

($I_{13}$) $\sum\f{1}{r_{i}}<\sum\f{1}{q_{i}}<\sum\f{1}{p_{i}}$:

We observe that $\sum\f{1}{r_{i}}<\sum\f{1}{p_{i}}$ and \eqref{mixed3.1}
mean that $\sum\f{1}{p_{i}}<\sum\f{1}{q_{i}}$. Therefore, we need not consider the case that $\sum\f{1}{r_{i}}<\sum\f{1}{q_{i}}<\sum\f{1}{p_{i}}$  here.

($I_{14}$) $\sum\f{1}{r_{i}}=\sum\f{1}{q_{i}}<\sum\f{1}{p_{i}}$:

 An easy computation yields that
\be\label{mixed3.3}
\sum \f{1}{p_{i}}=\sum\f{ \theta}{q_{i}}+\sum \f{1-\theta}{r_{i}}- [(1-\theta)s-\sigma]
<\sum\f{1}{q_{i}}+(1-\theta)\sum \f{1}{r_{i}}=\sum \f{1}{q_{i}}.
\ee
It turns out that $\sum\f1p_{i}<\sum\f1q_{i}=\sum\f1r_{i}$. This is a contradiction.

($I_{15}$) $\sum\f{1}{q_{i}}<\sum\f{1}{r_{i}}<\sum\f{1}{p_{i}}$:

We assert that $\sum\f{1}{q_{i}}<\sum\f{1}{r_{i}}$ and  $\sum\f{1}{q_{i}} <\sum\f{1}{p_{i}}$ yield a contradiction.
Note that
\be\label{wwkey}
\B(\f{1-\f{\sigma}{s}-\theta}{1-\f{\sigma}{s}}\B)\B[\sum\f1p_{i}-(\sum\f1r_{i}-( {s-\sigma} ))\B]=
\B(\f{ \theta}{1-\f{\sigma}{s}}\B)\B[ (1-\f{\sigma}{s})
\sum\f1q_{i}+\sum\f{ \sigma}{sr_{i}}-\sum\f1p_{i}\B].
\ee
When $(1-\f{\sigma}{s})
\sum\f1q_{i}+\sum\f{ \sigma}{sr_{i}}=\sum\f1p_{i}$, we see that $\sum\f1q_{i}=\sum\f1r_{i}-s$. This contradicts the hypothesis that $s\neq \sum(\f1r_{i}-\f1q_{i})$.
As a conseuqence, we only need to consider the subcase that $(1-\f{\sigma}{s})
\sum\f1q_{i}+\sum\f{ \sigma}{sr_{i}}<\sum\f1p_{i}$ or $(1-\f{\sigma}{s})
\sum\f1q_{i}+\sum\f{ \sigma}{sr_{i}}>\sum\f1p_{i}$.
On one hand, we derive from  $\sum\f{1}{q_{i}}<\sum\f{1}{r_{i}}$ and $\sum\f1p_{i}<(1-\f{\sigma}{s})
\sum\f1q_{i}+\sum\f{ \sigma}{sr_{i}} $ that $\sum\f1p_{i}<\sum\f{1}{r_{i}}$, which leads to a
contradiction. On the other hand,  if there hold
  $\sum\f{1}{q_{i}}>\sum\f{1}{r_{i}}$ and $\sum\f1p_{i}<(1-\f{\sigma}{s})
\sum\f1q_{i}+\sum\f{ \sigma}{sr_{i}} $, then it follows from \eqref{wwkey} that
$\sum\f1p_{i}-(\sum\f1r_{i}-( {s-\sigma} ))<0$, namely,
$0<{s-\sigma}<\sum\f1r_{i}-\sum\f1p_{i}<0$.  We get a contradiction once again.

($I_{2}$) $\sum\f{1}{r_{i}}=\sum\f{1}{p_{i}}$:

($I_{21}$)$\sum\f{1}{r_{i}}=\sum\f{1}{p_{i}}<\sum\f{1}{q_{i}}$:

By means of \eqref{addc}, we
observe that
\be\label{5191}
\|\dot{\Delta}_{j}u\|_{L^{\overrightarrow{ p}  }(\mathbb{R}^{n})}\leq C
\|\dot{\Delta}_{j}u\|^{1-\alpha}_{L^{\overrightarrow{q} }(\mathbb{R}^{n})}
\|\dot{\Delta}_{j}u\|^{\alpha}_{L^{(1+\varepsilon) \overrightarrow{ p},\infty}(\mathbb{R}^{n})},
~~\f{1}{p_{i}}=\f{1-\alpha}{q_{i}}+\f{\alpha}{(1+\varepsilon)p_{i}},
\ee
where $\varepsilon>0$ will be determined later.
We derive from  this,  Bernstein inequalities in mixed Lebesgue spaces and ${s-\sigma} +\sum\f{1}{p_{i}}-\sum\f{ 1}{r_{i}}>0$ that
\be\label{3.171}\ba
\|u\|_{\dot{B}^{\sigma}_{\overrightarrow{p}, 1}}&=\sum_{j\leq k}2^{j\sigma}\|\dot{\Delta}_{j}u\|_{L^{\overrightarrow{p} }(\mathbb{R}^{n})}+\sum_{j> k}2^{j\sigma}\|\dot{\Delta}_{j}u\|_{L^{\overrightarrow{p} }(\mathbb{R}^{n})}\\
&\leq C\sum_{j\leq k}2^{j[\sigma+\sum(\f1q_{i}-\f1p_{i})]}\|\dot{\Delta}_{j}u\|_{L^{q,\infty}(\mathbb{R}^{n})}+C\sum_{j> k}2^{j\sigma}\|\dot{\Delta}_{j}u\|^{1-\alpha}_{L^{\overrightarrow{q} }(\mathbb{R}^{n})}
\|\dot{\Delta}_{j}u\|^{\alpha}_{L^{(1+\varepsilon)\overrightarrow{p} }(\mathbb{R}^{n})}\\
&\leq C\f{2^{k[\sigma+\sum(\f1q_{i}-\f1p_{i})]}}{2^{ [\sigma+n(\f1r-\f1p)]}-1} \| u\|_{\dot{B}^{0}_{\overrightarrow{q},\infty }}+ C\sum_{j> k}2^{-j[s\alpha-\sum \f{ \varepsilon\alpha}{(1+\varepsilon)p_{i}}-\sigma]} \| u\|_{\dot{B}^{0}_{\overrightarrow{q} \infty}}^{\alpha} \| u\|_{\dot{B}^{s}_{\overrightarrow{r}, \infty}}^{1-\alpha},
\ea\ee
where we have used the fact that $\sigma+\sum(\f1q_{i}-\f1p_{i})>0$. \\
 Denote $\delta(\varepsilon)=s\alpha-\sum \f{ \varepsilon\alpha}{(1+\varepsilon)p_{i}}-\sigma$. From \eqref{5191}, we see that
 $$\delta(\varepsilon)= \f{s[(1+\varepsilon)p_{i}-(1+\varepsilon)q_{i}]}{(1+\varepsilon)(p_{i}-q_{i})}-\sum \f{ [\varepsilon(1+\varepsilon)p_{i}-(1+\varepsilon)q_{i}]}{(1+\varepsilon)p_{i}(1+\varepsilon)(p_{i}-q_{i})}-\sigma,$$
 and $\delta(\varepsilon)$ is a continuous function at neighborhood of $0$. Since $\delta(0)>0$, there exists a sufficiently small $\varepsilon>0$ such that $\delta(\varepsilon)>0$.
 Then it follows from \eqref{3.171} that
 \be\ba
\|u\|_{\dot{B}^{\sigma}_{\overrightarrow{p}, 1}}&
&\leq C\f{2^{k[\sigma+\sum(\f1q_{i}-\f1p_{i})]}}{2^{ [\sigma+\sum(\f1r_{i}-\f1p_{i})]}-1} \| u\|_{\dot{B}^{0}_{q,\infty,\infty}}+ C\f{-k[s\alpha- \sum\f{ \varepsilon\alpha}{(1+\varepsilon)p_{i}}-\sigma]}{1-2^{s\alpha- \sum\f{ \varepsilon\alpha}{(1+\varepsilon)p_{i}}-\sigma}}  \| u\|_{\dot{B}^{0}_{\overrightarrow{q}, \infty}}^{\alpha} \| u\|_{\dot{B}^{s}_{\overrightarrow{r}, \infty}}^{1-\alpha},
\ea\ee
 which implies that
$$\ba
\|u\|_{\dot{B}^{\sigma}_{\overrightarrow{p}, 1}} &\leq C  \| u\|_{\dot{B}^{0}_{\overrightarrow{q}, \infty}} ^{\theta} \| u\|_{\dot{B}^{s}_{\overrightarrow{r} ,\infty}}^{1-\theta}.
\ea$$

($I_{22}$)$\sum\f{1}{r_{i}}=\sum\f{1}{p_{i}}=\sum\f{1}{q_{i}}$:

From \eqref{mixed3.3}, we know that this subcase will not happen.

($I_{23}$)$\sum\f{1}{q_{i}}<\sum\f{1}{r_{i}}=\sum\f{1}{p_{i}}$:

This subcase is similar to the subcase ($I_{15}$). We omit the details here.

($I_{3}$)
$\sum\f{1}{p_{i}} <\sum\f{1}{r_{i}}$:

($I_{31}$)
$\sum\f{1}{p_{i}} <\sum\f{1}{r_{i}} =\sum\f{1}{q_{i}}$:

According to  the Bernstein inequality in Lemma \ref{lem2.1},  $\sigma+\sum (\f1r_{i}-\f1p_{i})>0$ and ${s-\sigma} +\sum\f{1}{p_{i}}-\sum\f{ 1}{r_{i}}>0$, we observe that
\be\ba\label{15.2}
\|u\|_{\dot{B}^{\sigma}_{\overrightarrow{p},1}}&=\sum_{j\leq k}2^{j\sigma}\|\dot{\Delta}_{j}u\|_{L^{\overrightarrow{p}}(\mathbb{R})^{n}}+\sum_{j> k}2^{j\sigma}\|\dot{\Delta}_{j}u\|_{L^{\overrightarrow{p}}(\mathbb{R})^{n}}\\
&\leq C\sum_{j\leq k}2^{j[\sigma+\sum (\f1r_{i}-\f1p_{i})]}\|\Delta_{j}u\|_{L^{\overrightarrow{r}}(\mathbb{R}^{n})}+C\sum_{j> k}2^{-j[s-\sigma-\sum (\f1r_{i}-\f1p_{i})]}2^{js}\|\dot{\Delta}_{j}u\|_{L^{\overrightarrow{r}}(\mathbb{R}^{n})}\\
&\leq C\f{2^{k[\sigma+\sum(\f1r_{i}-\f1p_{i})]}}{2^{ [\sigma+\sum(\f1r_{i}-\f1p_{i})]}-1} \| u\|_{\dot{B}^{0}_{\overrightarrow{r},\infty}}+C\f{ 2^{-k[s-\sigma-\sum(\f1r_{i}-\f1p_{i})]}}{1-2^{- [s-\sigma-\sum(\f1r_{i}-\f1p_{i})]}}  \| u\|_{\dot{B}^{s}_{\overrightarrow{r},\infty}}.
\ea\ee
Thus the desired inequality
\be\ba\label{desriinequ}
\|u\|_{\dot{B}^{\sigma}_{p,1,1}} &\leq C  \| u\|_{\dot{B}^{0}_{\overrightarrow{r},\infty}} ^{\theta} \| u\|_{\dot{B}^{s}_{\overrightarrow{r},\infty}}^{1-\theta}
\ea\ee
follows by taking
$$ {2^{k[\sigma+\sum(\f1r_{i}-\f1p_{i})]}}  \| u\|_{\dot{B}^{0}_{\overrightarrow{r},\infty}}\approx  { 2^{-k[s-\sigma-\sum(\f1r_{i}-\f1p_{i})]}} \| u\|_{\dot{B}^{s}_{\overrightarrow{r},\infty,\infty}}.$$

($I_{32}$) $\sum\f{1}{p_{i}} <\sum\f{1}{r_{i}} <\sum\f{1}{q_{i}}$:

  We proceed as in the proof of \eqref{15.2} and get
$$\ba
\|u\|_{\dot{B}^{\sigma}_{\overrightarrow{p} ,1}}&=\sum_{j\leq k}2^{j\sigma}\|\dot{\Delta}_{j}u\|_{L^{\overrightarrow{p} }(\mathbb{R}^{n})}+\sum_{j> k}2^{j\sigma}\|\dot{\Delta}_{j}u\|_{L^{\overrightarrow{p} }(\mathbb{R}^{n})}\\
&\leq C\sum_{j\leq k}2^{j[\sigma+\sum(\f1q_{i}-\f1p_{i})]}\|\dot{\Delta}_{j}u\|_{L^{\overrightarrow{q}}(\mathbb{R}^{n})}+C\sum_{j> k}2^{-j[s-\sigma-\sum(\f1r_{i}-\f1p_{i})]}2^{js}\|\dot{\Delta}_{j}u\|_{L^{r}(\mathbb{R}^{n})}\\
&\leq C\f{2^{k[\sigma+\sum(\f1q_{i}-\f1p_{i})]}}{2^{ [\sigma+\sum(\f1q_{i}-\f1p_{i})]}-1} \| u\|_{\dot{B}^{0}_{q,\infty}}+C\f{ 2^{-k[s-\sigma-\sum(\f1r_{i}-\f1p_{i})]}}{1-2^{- [s-\sigma-\sum(\f1r_{i}-\f1p_{i})]}}  \| u\|_{\dot{B}^{s}_{\overrightarrow{r} \infty}},
\ea$$
where we have used the fact that $\sigma+\sum(\f1r_{i}-\f1p_{i})>0$ and ${s-\sigma} +\sum\f{1}{p_{i}}-\sum\f{ 1}{r_{i}}>0$.\\
Consequently, as the derivation of \eqref{desriinequ},   we find out that
$$\ba
\|u\|_{\dot{B}^{\sigma}_{p,1,1}} &\leq C  \| u\|_{\dot{B}^{0}_{\overrightarrow{q}}} ^{\theta} \| u\|_{\dot{B}^{s}_{\overrightarrow{r}, \infty}}^{1-\theta}.
\ea$$

($I_{33}$) $\sum\f{1}{p_{i}} <\sum\f{1}{q_{i}}<\sum\f{1}{r_{i}} $:

 Taking advantage of Bernstein inequalities in Lemma  \ref{lem2.1} and ${s-\sigma} +\sum\f{1}{p_{i}}-\sum\f{ 1}{r_{i}}>0$, we know that
$$\ba
\|u\|_{\dot{B}^{\sigma}_{\overrightarrow{p},1 }}&=\sum_{j\leq k}2^{j\sigma}\|\dot{\Delta}_{j}u\|_{L^{\overrightarrow{p} }(\mathbb{R}^{n})}+\sum_{j> k}2^{j\sigma}\|\dot{\Delta}_{j}u\|_{L^{\overrightarrow{p} }(\mathbb{R}^{n})}\\
&\leq \sum_{j\leq k}2^{j[\sigma+\sum(\f1q_{i}-\f1p_{i})]}\|\dot{\Delta}_{j}u\|_{\overrightarrow{q}(\mathbb{R}^{n})}+\sum_{j> k}2^{-j[s-\sigma-\sum(\f1r_{i}-\f1p_{i})]}2^{js}\|\dot{\Delta}_{j}u\|_{L^{\overrightarrow{r } }(\mathbb{R}^{n})}\\
&\leq \f{2^{k[\sigma+\sum(\f1q_{i}-\f1p_{i})]}}{2^{ [\sigma+\sum(\f1q_{i}-\f1p_{i})]}-1} \| u\|_{\dot{B}^{0}_{q,\infty}}+\f{ 2^{-k[s-\sigma-\sum(\f1r_{i}-\f1p_{i})]}}{1-2^{- [s-\sigma-\sum(\f1r_{i}-\f1p_{i})]}}  \| u\|_{\dot{B}^{s}_{\overrightarrow{r}, \infty}},
\ea$$
where we have used the fact that $\sigma+\sum(\f1r_{i}-\f1p_{i})>0$.\\
Therefore,   it follows from
 $$\f{2^{k[\sigma+\sum(\f1q_{i}-\f1p_{i})]}}{2^{ [\sigma+\sum(\f1q_{i}-\f1p_{i})]}-1} \| u\|_{\dot{B}^{0}_{q,\infty}}\approx\f{ 2^{-k[s-\sigma-\sum(\f1r_{i}-\f1p_{i})]}}{1-2^{- [s-\sigma-\sum(\f1r_{i}-\f1p_{i})]}}  \| u\|_{\dot{B}^{s}_{\overrightarrow{r}, \infty}}$$
that
$$\ba
\|u\|_{\dot{B}^{\sigma}_{p,1,1}} &\leq C\B(\f{1}{\sigma+\sum(\f1q_{i}-\f1p_{i})}+\f{1}{s-\sigma-\sum(\f1r_{i}-\f1p_{i})}\B) \| u\|_{\dot{B}^{0}_{\overrightarrow{r},\infty}} ^{\theta} \| u\|_{\dot{B}^{s}_{\overrightarrow{r}, \infty}}^{1-\theta}.
\ea$$

($I_{34}$) $ \sum\f1q_{i}=\sum \f1p_{i}<\f1r_{i}$:

 It is worth remarking  that this subcase implies that $\sigma>0$. In the same manner as   \eqref{5191}, we see that
\be
\|\dot{\Delta}_{j}u\|_{L^{\overrightarrow{p} }(\mathbb{R}^{n})}\leq C
\|\dot{\Delta}_{j}u\|^{1-\alpha}_{L^{\overrightarrow{r }}(\mathbb{R}^{n})}
\|\dot{\Delta}_{j}u\|^{\alpha}_{L^{(1+\varepsilon)\overrightarrow{p},\infty}(\mathbb{R}^{n})},
~~\f{1}{p_{i}}=\f{1-\alpha}{q_{i}}+\f{\alpha}{(1+\varepsilon)p_{i}},
\ee
where $\varepsilon>0$ will be determined later.
$$\ba
\|u\|_{\dot{B}^{\sigma}_{\overrightarrow{p} ,1}}
&\leq C\sum_{j\leq k}2^{j\sigma}\|\dot{\Delta}_{j}u\|^{1-\alpha}_{L^{\overrightarrow{r } }(\mathbb{R}^{n})}
\|\Delta_{j}u\|^{\alpha}_{L^{(1+\varepsilon)\overrightarrow{p} }(\mathbb{R}^{n})}+C\sum_{j> k}2^{-j[s-\sigma-\sum(\f1r_{i}-\f1p_{i})]}2^{js}\|\dot{\Delta}_{j}u\|_{L^{\overrightarrow{r } }(\mathbb{R}^{n})}\\
&\leq  C\sum_{j\leq k}2^{j[\sigma+\sum\f{ \varepsilon\alpha}{p_{i}(1+\varepsilon)}-s(1-\alpha)]}  \| u\|_{\dot{B}^{0}_{\overrightarrow{q}, \infty}}^{\alpha} \| u\|_{\dot{B}^{s}_{\overrightarrow{r } ,\infty}}^{1-\alpha}+C\f{ 2^{-k[s-\sigma-\sum(\f1r_{i}-\f1p_{i})]}}{1-2^{- [s-\sigma-\sum(\f1r_{i}-\f1p_{i})]}}  \| u\|_{\dot{B}^{s}_{\overrightarrow{r}, \infty}},
\ea$$
As the arguments in ($I_{21}$), we can choose $\varepsilon>0$ sufficiently small to ensure that
$\sigma+\f{n\varepsilon\alpha}{p(1+\varepsilon)}-s(1-\alpha)>0$ and get the desired inequality. We omit the details.

($I_{35}$) $\sum\f{1}{q_{i}} <\sum\f{1}{p_{i}}<\sum\f{1}{r_{i}} $:

 Notice that $\sum\f{1}{r_{i}}+\sigma-s< \sum\f{1}{p_{i}}$ and  \eqref{wwkey} mean that
\be\label{3.19}
\sum\f{ 1}{p_{i}}
 < (1-\f{\sigma}{s})\sum\f{1}{q_{i}}+\f{\sigma}{s}\sum\f{1}{r_{i}}.
\ee
This together with \eqref{addc} leads to
\be\label{45.1}
\sigma-s(1-\alpha)>0
\ee
and
\be\label{5.2}
\|\dot{\Delta}_{j}u\|_{L^{\overrightarrow{p}}(\mathbb{R}^{n})}
\leq
\|\dot{\Delta}_{j}u
\|^{\alpha}_{L^{\overrightarrow{q} }(\mathbb{R}^{n})}
\|\dot{\Delta}_{j}u\|^{1-\alpha}_{L^{\overrightarrow{r}}(\mathbb{R}^{n})}.
\ee
By means of Bernstein inequalities, \eqref{5.2}, \eqref{45.1} and   $\sum\f{1}{r_{i}}+\sigma-s< \sum\f{1}{p_{i}}$, we infer  that
$$\ba
\|u\|_{\dot{B}^{\sigma}_{\overrightarrow{p},1}}&=\sum_{j\leq k}2^{j\sigma}\|\dot{\Delta}_{j}u\|_{L^{\overrightarrow{p}}(\mathbb{R}^{n})}+\sum_{j> k}2^{j\sigma}\|\dot{\Delta}_{j}u\|_{L^{\overrightarrow{p}}(\mathbb{R}^{n})}\\
&\leq \sum_{j\leq k}2^{j\sigma}\|\dot{\Delta}_{j}u
\|^{\alpha}_{L^{\overrightarrow{q} }(\mathbb{R}^{n})}
\|\dot{\Delta}_{j}u\|^{1-\alpha}_{L^{\vec{r }}(\mathbb{R}^{n})} +\sum_{j> k}2^{-j[s-\sigma-\sum(\f1r_{i}-\f1p_{i})]}2^{js}\|\dot{\Delta}_{j}u
\|_{L^{\overrightarrow{r}}(\mathbb{R}^{n})}\\
&\leq C\f{ 2^{ k[\sigma-s(1-\alpha) ]}}{2^{  [\sigma-s(1-\alpha) ]}-1} \| u\|_{\dot{B}^{0}_{\overrightarrow{q} }}^{\alpha} \| u\|_{\dot{B}^{s}_{\overrightarrow{r} ,\infty}}^{1-\alpha}+C\f{ 2^{-k[s-\sigma-\sum(\f1r_{i}-\f1p_{i})]}}{1-2^{- [s-\sigma-\sum(\f1r_{i}-\f1p_{i})]}}  \| u\|_{\dot{B}^{s}_{r_{i}, \infty} }.
\ea$$
Therefore we may deduce that
 $$\ba
\|u\|_{\dot{B}^{\sigma}_{\overrightarrow{p},1 }} &\leq C\B(\f{1}{\sigma-s(1-\alpha)}+\f{1}{s-\sigma-\sum(\f1r_{i}-\f1p_{i})}\B) \| u\|_{\dot{B}^{0}_{\overrightarrow{r},\infty}} ^{\theta} \| u\|_{\dot{B}^{s}_{\overrightarrow{r}, \infty}}^{1-\theta}.
\ea$$

{\bf Case 2:} $\sum\f1q_{i}+ {\sigma} <\sum\f1p_{i}$  and  ${s-\sigma} +\sum\f{1 }{p_{i}}-\sum\f{ 1}{r_{i}}<0$.\\
  It is clear that  $\sum\f1q_{i}+ {\sigma} <\sum\f1p_{i}<\sum\f{ 1}{r_{i}}$.
 From $\sum\f{1}{r_{i}}+\sigma-s> \sum\f{1}{p_{i}}$ and \eqref{wwkey}, we see that
$$\sum\f{ 1}{p_{i}}
 >\sum(1-\f{\sigma}{s})\f{1}{q_{i}}+\f{\sigma}{s}\sum\f{1}{r_{i}},$$
 which  together with \eqref{addc} leads to
 \be\label{45.5}
 s(1-\alpha)-\sigma >0
 \ee
and
\be\label{5.412}
\|\dot{\Delta}_{j}u\|_{L^{\overrightarrow{p} }(\mathbb{R}^{n})}\leq\|\dot{\Delta}_{j}u
\|^{\alpha}_{L^{\overrightarrow{q}}(\mathbb{R}^{n})}\|\dot{\Delta}_{j}u
\|^{1-\alpha}_{L^{\overrightarrow{r}}(\mathbb{R}^{n})}.
\ee
In view of Bernstein inequalities, \eqref{5.412}, ${s-\sigma} +\sum\f{1 }{p_{i}}-\sum\f{ 1}{r_{i}}<0$ and \eqref{45.5}, we have
 $$\ba
\|u\|_{\dot{B}^{\sigma}_{\overrightarrow{p},1}}&=\sum_{j\leq k}2^{j\sigma}\|\dot{\Delta}_{j}u\|_{L^{\overrightarrow{p} }(\mathbb{R}^{n})}+\sum_{j> k}2^{j\sigma}\|\dot{\Delta}_{j}u\|_{L^{\overrightarrow{p} }(\mathbb{R}^{n})}\\
&\leq \sum_{j\leq k} 2^{ j[\sigma-s+\sum(\f1r_{i}-\f1p_{i})]}2^{js}\|\dot{\Delta}_{j}u\|_{L^{\overrightarrow{r }}(\mathbb{R}^{n})}+
\sum_{j> k}2^{j\sigma}\|\dot{\Delta}_{j}u
\|^{\alpha}_{L^{\overrightarrow{q}}(\mathbb{R}^{n})}\|\dot{\Delta}_{j}u
\|^{1-\alpha}_{L^{\overrightarrow{r}}(\mathbb{R}^{n})}
\\
&\leq \f{ 2^{ k[   \sigma-s+\sum(\f1r_{i}-\f1p_{i})]}}{ 2^{- [\sigma -s+\sum(\f1r_{i}-\f1p_{i})]-1}}  \| u\|_{\dot{B}^{s}_{\overrightarrow{r}, \infty}} +\f{ 2^{-k[s(1-\alpha)-\sigma ]}}{1-2^{- [s(1-\alpha)-\sigma ]}} \| u\|_{\dot{B}^{0}_{\overrightarrow{q},\infty}}^{\alpha} \| u\|_{\dot{B}^{s}_{\overrightarrow{r}, \infty}}^{1-\alpha}.
\ea$$
Repeating the  deduction of \eqref{desriinequ}, we know that
 $$\ba
\|u\|_{\dot{B}^{\sigma}_{\overrightarrow{p},1 }} &\leq C\B( \f{1}{s-\sigma-\sum(\f1r_{i}-\f1p_{i})}+\f{1}{s(1-\alpha)-\sigma}\B) \| u\|_{\dot{B}^{0}_{\overrightarrow{q},\infty}} ^{\theta} \| u\|_{\dot{B}^{s}_{\overrightarrow{r}, \infty}}^{1-\theta}.
\ea$$

{\bf Case 3:}  $\sum\f1q_{i}+ {\sigma} =\sum\f1p_{i}$  and  ${s-\sigma} +\sum\f{1}{p_{i}}-\sum\f{ 1}{r_{i}}=0$. This case  will not be considered in this theorem, owing to the hypothesis that $s\neq \sum(\f1r_{i}-\f1q_{i})$.

At this stage, we complete the proof of Theorem \ref{the1.2}.
\end{proof}

\section{An
application of Gagliardo-Nirenberg inequality in anisotropic Lebesgue spaces}
This section is    concerned with  the application  of Gagliardo-Nirenberg inequality \eqref{glxied1} in anisotropic Lebesgue spaces to the
   energy conservation criteria of 3D Navier-Stokes equations.
 \begin{proof}[Proof of Theorem \ref{the1.3}]
(1) $v\in L^{\f{2p}{p-1}}(0,T;\,L^{\f{2q_{1}}{q_{1}-1}}L^{\f{2q_{2}}{q_{2}-1}}L^{\f{2q_{3}}{q_{3}-1}}(\mathbb{R}^{3}))$, $\nabla v \in L^{p} (0, T ; L^{\overrightarrow{q}} (\mathbb{R}^{3} ) )$, $1<p,\overrightarrow{q}\leq\infty.$

For $N\in\mathbb{N},$ applying nonhomogeneous Littlewood-Paley operator $S_{N}$ to the Navier-Stokes equations \eqref{NS} and taking the $L^2$ inner product with $S_{N}v$, one arrives at
\be\label{zheng1}
\f12\|S_{N}v(t)\|_{L^{2}(\mathbb{R}^{3})}^{2}+\int_{0}^{t}\|\nabla S_{N}v\|_{L^{2}(\mathbb{R}^{3})}^{2}ds =\f12\|S_{N}v_{0}\|_{L^{2}(\mathbb{R}^{3})}^{2}-\int_{0}^{t}\int_{\mathbb{R}^3}
\partial_{j}(v_{j}v_{i})S^{2}_{N}v_{i}\,dxds.
\ee
Thus, in order to prove energy equality
\be\label{EI}
 \|v(t)\|_{L^{2}(\mathbb{R}^{3})}^{2}+2 \int_{0}^{t }\|\nabla v\|_{L^{2}(\mathbb{R}^{3})}^{2}ds= \|v_0\|_{L^{2}(\mathbb{R}^{3})}^{2}
\ee
for all $t\in[0,T)$, it suffices to  show that
$$\lim_{N\rightarrow\infty}\int_{0}^{t}\int_{\mathbb{R}^3}
\partial_{j}(v_{j}v_{i})S^{2}_{N}v_{i}\,dxds=0$$
by estimating  the nonlinear term in \eqref{zheng1}.

Indeed, it follows from the divegence-free condition and integration  by parts  that
\begin{equation*}
\begin{split}
&-\int_{0}^{t}\int_{\mathbb{R}^3}
\partial_{j}(v_{j}v_{i})S^{2}_{N}v_{i}\,dxds\\
=&-\int_{0}^{t}\int_{\mathbb{R}^3}
\partial_{j}(v_{j}S^{2}_{N}v_{i})S^{2}_{N}v_{i}\,dxds
-\int_{0}^{t}\int_{\mathbb{R}^3}
\partial_{j}[v_{j}(\text{I}_{\textnormal d}-S^{2}_{N})v_{i}]S^{2}_{N}v_{i}\,dxds\\
=&-\f12\int_{0}^{t}\int_{\mathbb{R}^3}
v_{j}\partial_{j}(S^{2}_{N}v_{i})^{2} \,dxds-\int_{0}^{t}\int_{\mathbb{R}^3}
\partial_{j}\big[v_{j}(\text{I}_{\textnormal d}-S^{2}_{N})v_{i}\big]S^{2}_{N}v_{i}\,dxds\\
=& \int_{0}^{t}\int_{\mathbb{R}^3}
S_{N}\big[v_{j}(\text{I}_{\textnormal d}-S^{2}_{N})v_{i}\big]S_{N}\partial_{j}v_{i}\,dxds,
\end{split}
\end{equation*}
where the notation $\text{I}_{\textnormal d}$ represents the identity operator. We further employ the H\"older inequality \eqref{HIAL} and Young's inequality \eqref{YoungI} in
anisotropic Lebesgue spaces to get
\be\ba
&\left|\int_{0}^{t}\int_{\mathbb{R}^3}S_{N}
\big[v_{j}(\text{I}_{\textnormal d}-S^{2}_{N})v_{i}\big]S_{N}\partial_{j}v_{i}\,dxds\right| \\
\leq&\big\| S_{N}\big[v_{j}(\text{I}_{\textnormal d}-S^{2}_{N})v_{i}\big]\big\|_{L^{\f{p}{p-1}}(0,T;
L^{\f{q_{1}}{q_{1}-1}}L^{\f{q_{2}}{q_{2}-1}}
L^{\f{q_{3}}{q_{3}-1}}(\mathbb{R}^{3}))}  \big\| S_{N}\partial_{j}v_{i} \big\|_{L^{p}(0,T;L^{\overrightarrow{q}}(\mathbb{R}^{3}))}\\
\leq&  C\big\|  v_{j}(\text{I}_{\textnormal d}-S^{2}_{N})v_{i}\big\|_{L^{\f{p}{p-1}}(0,T;L^{\f{q_{1}}{q_{1}-1}}L^{\f{q_{2}}{q_{2}-1}}
L^{\f{q_{3}}{q_{3}-1}}(\mathbb{R}^{3}))}  \big\| \partial_{j}v_{i} \big\|_{L^{p}(0,T;L^{\overrightarrow{q}}(\mathbb{R}^{3}))}\\
\leq& C\|  v   \|_{L^{\f{2p}{p-1}}(0,T;L^{\f{2q_{1}}{q_{1}-1}}L^{\f{2q_{2}}{q_{2}-1}}
L^{\f{2q_{3}}{q_{3}-1}}(\mathbb{R}^{3}))} \big\|\big(\text{I}_{\textnormal d}+S_{N}\big)\big(\text{I}_{\textnormal d}-S_{N}\big)v \big\|_{L^{\f{2p}{p-1}}(0,T;L^{\f{2q_{1}}{q_{1}-1}}L^{\f{2q_{2}}{q_{2}-1}}
L^{\f{2q_{3}}{q_{3}-1}}(\mathbb{R}^{3}))}\\& \times \| \nabla  v \|_{L^{p}(0,T;L^{\overrightarrow{q}}(\mathbb{R}^{3}))}\\
\leq& C\|  v   \|_{L^{\f{2p}{p-1}}(0,T;L^{\f{2q_{1}}{q_{1}-1}}L^{\f{2q_{2}}{q_{2}-1}}
L^{\f{2q_{3}}{q_{3}-1}}(\mathbb{R}^{3}))} \big\|\big(\text{I}_{\textnormal d}-S_{N}\big)v \big\|_{L^{\f{2p}{p-1}}(0,T;L^{\f{2q_{1}}{q_{1}-1}}L^{\f{2q_{2}}{q_{2}-1}}
L^{\f{2q_{3}}{q_{3}-1}}(\mathbb{R}^{3}))} \\& \times \| \nabla  v \|_{L^{p}(0,T;L^{\overrightarrow{q}}(\mathbb{R}^{3}))}.
\ea\ee
In the light of Lemma \ref{lem2.3}, we see that, as $N\rightarrow\infty$,
$$\big\|\big(\text{I}_{\textnormal d}-S_{N}\big)v \big\|_{L^{\f{2p}{p-1}}(0,T;\,L^{\f{2q_{1}}{q_{1}-1}}L^{\f{2q_{2}}{q_{2}-1}}
L^{\f{2q_{3}}{q_{3}-1}}
(\mathbb{R}^{3}))}\rightarrow0.
$$
It turns out that, as $N\rightarrow\infty$,
$$
-\int_{0}^{t}\int_{\mathbb{R}^3}
\partial_{j}(v_{j}v_{i})S^{2}_{N}v_{i}\,dxds\rightarrow0.
$$
Hence, letting $N\rightarrow\infty$ in    \eqref{zheng1}, we conclude the energy equality   \eqref{EI}.
The proof of this part is completed.

 (2)   $v\in L^{p}(0,T;L^{\overrightarrow{q}}(\mathbb{R}^{3})),~\text{with}~\f{1}{p}+
 \f{1}{q_{1}}+
 \f{1}{q_{2}}+
 \f{1}{q_{3}}=1~\text{and}~~ \sum_{i=1}^{3}\frac{1}{q_{i}}\leq1,~1<\overrightarrow{q}\leq4.$

Making use of the Gagliardo-Nirenberg inequality \eqref{glxied1} in mixed norm spaces, we
obtain
$$
\|v\|_{L^{4}(\mathbb{R}^{3})} \leq C\|\nabla v\|_{L^{ 2}(\mathbb{R}^{3})}^{\f{ \sum\f{1}{q_{i}}-\f34 }{\sum\f{1}{q_{i}}-\f12}}\|v\|_{L^{\overrightarrow{q}}(\mathbb{R}^{3})}^{\f{\f14}{\sum\f{1}{q_{i}}-\f12}},
$$
where we have used the fact that $\sum\frac{1}{q_{i}}\geq  \f34$ and $\f14\leq \f12(\f{ \sum\f{1}{q_{i}}-\f34 }{\sum\f{1}{q_{i}}-\f12})+\f{1}{q_{j}}(\f{\f14}{\sum\f{1}{q_{i}}-\f12})$ for $j=1,2,3$.

Since $\sum_{i=1}^{3}\f{1}{q_{i}} \leq 1$, we further conclude by the H\"older inequality  that
$$\ba
\int_{0}^{T}\|v\|_{L^4(\mathbb{R}^{3})}^{4} d t \leq&C \int_{0}^{T}\|\nabla v\|_{L^{ 2}(\mathbb{R}^{3})}^{\f{ \sum\f{4}{q_{i}}-3 }{\sum\f{1}{q_{i}}-\f12}}\|v\|_{L^{\overrightarrow{q}}(\mathbb{R}^{3})}^{\f{1}{\sum\f{1}{q_{i}}-\f12}}d t\\
\leq & C\B(\int_{0}^{T}\|\nabla v\|^{2}_{L^{ 2}(\mathbb{R}^{3})}d t\B)^{\f{2(\sum\f{1}{q_{i}}-\f34)}{\sum\f{1}{q_{i}}-\f12}}\B(\int_{0}^{T}
\|v\|^{\f{1}{1-\sum\f{1}{q_{i}}}}_{L^{\overrightarrow{q}}(\mathbb{R}^{3})}
  d t\B)^{\frac{1-\sum\f{1}{q_{i}}}{\sum\f{1}{q_{i}}-\f12}},
\ea$$
which leads to  $v\in L^{4}(0,T;L^{4}(\mathbb{R}^{3})).$  The Lions's class for energy equality of weak solutions to Navier-Stokes equations allows us to  complete the proof of this part.

(3) $\nabla v \in L^{p} (0, T ; L^{\overrightarrow{q}} (\mathbb{R}^{3}) ),~\text{with}~
\frac  {1}{p}+\f{1}{q_{1}}+
 \f{1}{q_{2}}+
 \f{1}{q_{3}} =2~\text{and}~ \sum_{i=1}^{3}\frac{1}{q_{i}}\leq2$, $~1<\overrightarrow{q}\leq \frac{9}{5}$.

According to the result of (1) in this theorem, it suffices to show that $v\in L^{\f{2p}{p-1}}(0,T;L^{\f{2q_{1}}{q_{1}-1}}L^{\f{2q_{2}}{q_{2}-1}}
L^{\f{2q_{3}}{q_{3}-1}}
(\mathbb{R}^{3})).$ Thanks to the Gagliardo-Nirenberg inequality \eqref{glxied1} and the Sobolev embedding \eqref{bpsobo} in anisotropic Lebesgue spaces, we have
$$\ba
\|v\|_{L^{\f{2q_{1}}{q_{1}-1}}L^{\f{2q_{2}}{q_{2}-1}}
L^{\f{2q_{3}}{q_{3}-1}}
(\mathbb{R}^{3})} \leq& C\|\nabla v\|_{L^{2}(\mathbb{R}^{3})}^{\f{3(\f{1}{q_{1}}+
 \f{1}{q_{2}}+
 \f{1}{q_{3}})-5}{2(\f{1}{q_{1}}+
 \f{1}{q_{2}}+
 \f{1}{q_{3}})-3 }}\| v\|_{L^{\f{3q_{1}}{3-q_{1}}}L^{\f{3q_{2}}{3-q_{2}}}L^{\f{3q_{3}}{3-q_{3}}}(\mathbb{R}^{3})}^{\f{ 2- (\f{1}{q_{1}}+
 \f{1}{q_{2}}+
 \f{1}{q_{3}})}{ 2(\f{1}{q_{1}}+
 \f{1}{q_{2}}+
 \f{1}{q_{3}})-3}}\\\leq& C\|\nabla v\|_{L^{2}(\mathbb{R}^{3})}^{\frac{p-3}{p-2}}\|\nabla v\|_{L^{\overrightarrow{q}}(\mathbb{R}^{3})}^{\frac{1}{p-2}},
\ea$$
which leads to
$$\ba
&\|v\|_{L^{\f{2p}{p-1}}(0,T;L^{\f{2q_{1}}{q_{1}-1}}L^{\f{2q_{2}}{q_{2}-1}}
L^{\f{2q_{3}}{q_{3}-1}}
(\mathbb{R}^{3}))}\leq  C\|\nabla v\|_{L^{2}(0,T;L^{2}(\mathbb{R}^{3}))}^{\frac{p-3}{p-2}}\|\nabla v\|_{L^{p}(0,T;L^{\overrightarrow{q}}(\mathbb{R}^{3}))}^{\frac{1}{p-2}}.
\ea$$
Here the restriction $1<\overrightarrow{q}\leq \frac{9}{5}$ guarantees that the second condition in  \eqref{condi}  is valid.
Hence, we finish the proof of this part.

(4) $\nabla v \in L^{p} (0, T ; L^{\overrightarrow{q}} (\mathbb{R}^{3}) ),~\text{with}~
\frac{1}{p}+\frac{2}{5  }(\f{1}{q_{1}}+
 \f{1}{q_{2}}+
 \f{1}{q_{3}})=1~\text{and}~\sum_{i=1}^{3}\frac{1}{q_{i}}\leq\f53,~1<\overrightarrow{q}\leq3.$

As mentioned above, in terms of the result of (1) in this theorem, it is enough to prove that $v\in L^{\f{2p}{p-1}}(0,T;L^{\f{2q_{1}}{q_{1}-1}}L^{\f{2q_{2}}{q_{2}-1}}
L^{\f{2q_{3}}{q_{3}-1}}
(\mathbb{R}^{3}))$. To this end, in view  of the Gagliardo-Nirenberg inequality \eqref{glxied1} in   anisotropic Lebesgue spaces,  we deduce that
$$
\|v\|_{L^{\f{2q_{1}}{q_{1}-1}}L^{\f{2q_{2}}{q_{2}-1}}
L^{\f{2q_{3}}{q_{3}-1}}
(\mathbb{R}^{3})}\leq C\|v\|_{L^{2}(\mathbb{R}^{3})}^{\f{5 -3(\f{1}{q_{1}}+
 \f{1}{q_{2}}+
 \f{1}{q_{3}})}{5 -2(\f{1}{q_{1}}+
 \f{1}{q_{2}}+
 \f{1}{q_{3}})}}\|\nabla v\|_{L^{\overrightarrow{q}}(\mathbb{R}^{3})}^{\f{ \f{1}{q_{1}}+
 \f{1}{q_{2}}+
 \f{1}{q_{3}} }{5 -2(\f{1}{q_{1}}+
 \f{1}{q_{2}}+
 \f{1}{q_{3}})} }=C\|v\|_{L^{2}(\mathbb{R}^{3})}^{\frac{3-p}{2}}\|\nabla v\|_{L^{\overrightarrow{q}}(\mathbb{R}^{3})}^{\frac{p-1}{2}},
$$
which implies that
$$
\|v\|_{L^{\f{2p}{p-1}}(0,T;L^{\f{2q_{1}}{q_{1}-1}}L^{\f{2q_{2}}{q_{2}-1}}
L^{\f{2q_{3}}{q_{3}-1}}
(\mathbb{R}^{3}))}\leq C\|v\|_{L^{\infty}(0,T;L^{2}(\mathbb{R}^{3}))}^{\frac{3-p}{2}}\|\nabla v\|_{L^{p}(0,T;L^{\overrightarrow{q}}(\mathbb{R}^{3}))}^{\frac{p-1}{2}}.
$$
At this stage, all proofs of this theorem are given.
\end{proof}

 \section*{Acknowledgements}
%The authors thank the anonymous referee and the associated editor for the invaluable
%comments and suggestions which helped to improve the paper greatly.

 Wang was partially supported by  the National Natural
 Science Foundation of China under grant (No. 11971446, No. 12071113   and  No.  11601492).
 Wei was partially supported by the National Natural Science Foundation of China under grant (No. 11601423, No. 11871057).

%\end{CJK*}
\end{document}